\documentclass{article}

\usepackage{amsmath,amsthm,amssymb,mathtools}  
\usepackage{dsfont}				% Símbolos para os conjuntos R, N, etc.

\usepackage{enumerate}
\usepackage{graphicx}			% Para incluir figuras
\usepackage{subcaption}			% subfiguras
\usepackage[font=small]{caption}	% para fazer captions com fonte pequena
\usepackage[latin1]{inputenc}   % Para poder colocar acentos diretamente
\usepackage{verbatim}		% \begin{comment} pra esconder texto

\usepackage[
disable,					% disable para acelerar e limpar Log
textsize=scriptsize,textwidth=4.2cm]{todonotes}	% margin notes
\newcommand{\MAYBE}[1]{\todo[color=orange!60]{#1}}  % notas leves
\newcommand{\SELF}[1]{\todo[color=green!40]{#1}} % note to myself
\newcommand{\OMIT}[1]{\todo[color=gray!30]{#1}}  % omitted details of proofs
\newcommand{\CITE}[1]{\todo[color=cyan!30]{#1}}  % citação extra

\newcommand{\SELFL}[1]{\reversemarginpar\todo[color=green!50]{#1}} 
\newcommand{\OMITL}[1]{\reversemarginpar\todo[color=gray!40]{#1}}  
\newcommand{\CITEL}[1]{\reversemarginpar\todo[color=green!20]{#1}}  

\newcommand{\OMITR}[1]{\normalmarginpar\todo[color=gray!40]{#1}}  

\usepackage[hyperindex,breaklinks]{hyperref}	% refs como links, \nameref e \cref
\hypersetup{colorlinks=true, linkcolor=blue, citecolor=blue}
\usepackage[nameinlink]{cleveref} 	% \cref, \Cref  (carregar depois do hyperref)
\usepackage{tikz-cd}			% diagrams
\usepackage{bm}					% bold greek letters in math
\usepackage{tensor}				% left/right sub/superscript  \tensor*[^a_b^c_d]{M}{^a_b^c_d} Sem * cada índice tem coluna determinada
\usepackage{cmll}				% \Perp, \simperp
\usepackage{stmaryrd}			% \llbracket, \rrbracket

\usepackage{makecell}			% cell with multiple lines in table
\usepackage{booktabs}			% makes tables look better
\renewcommand{\arraystretch}{2}	% more space between table lines

\newtheorem{theorem}{Theorem}[section]
\newtheorem{proposition}[theorem]{Proposition}
\newtheorem{corollary}[theorem]{Corollary}

\theoremstyle{definition}
\newtheorem{definition}[theorem]{Definition} 

\newtheorem{example}[theorem]{Example}

\newtheoremstyle{named}{}{}{\itshape}{}{\bfseries}{.}{.5em}{\thmnote{#3}}
\theoremstyle{named}

\def\N			{\mathds{N}}
\def\Z			{\mathds{Z}}
\def\R			{\mathds{R}}

\def\Id			{\mathds{1}}

\def\PP			{\mathds{P}}
\def\VV			{\mathcal{V}}
\def\WW			{\mathcal{W}}
\def\II			{\mathcal{I}}		% increasing multi-indices set
\def\MM			{\mathcal{M}}		% multi-indices set
\def\AA			{\mathcal{A}}

\def\ii			{\mathbf{i}}				
\def\jj			{\mathbf{j}}				
\def\kk			{\mathbf{k}}				
\def\ll			{\mathbf{l}}				
\def\mm			{\mathbf{m}}				
\def\aa			{\mathbf{a}}				
\def\bb			{\mathbf{b}}				
\def\cc			{\mathbf{c}}				
\def\dd			{\mathbf{d}}				
\def\ee			{\mathbf{e}}				
\def\rr			{\mathbf{r}}				
\def\ss			{\mathbf{s}}

\def\xx			{\mathbf{x}}				
\def\yy			{\mathbf{y}}

\def\ie						{i.e.\ }
\def\resp					{resp.\ }

\newcommand\inner[1] 		{\langle #1 \rangle}		% inner product
\newcommand\pairs[2] 		{|#1>#2|}					% (i,j) conta pares j<i
\newcommand\ord[1]			{\overline{#1}} %{[#1]}						% ordering operator
\newcommand\scom[1] 		{\llbracket #1 \rrbracket}	% supercommutator
\newcommand\sacom[1] 		{\{\mskip-6mu\{ #1 \}\mskip-6mu\}}	% superanticommutator
\newcommand\isp[1] 			{\lfloor #1 \rfloor}		% inner space
\newcommand\osp[1] 			{\lceil #1 \rceil}			% outer space
\def\lcontr					{\mathbin{\lrcorner}}		% left contraction
\def\rcontr					{\mathbin{\llcorner}}		% right contraction
\def\PP						{\mathcal{P}}				% orthogonal projection in exterior algebra
\newcommand{\lH}[2][\star]	{\tensor[^{#1}]{{#2}}{}}	% left Hodge dual
\newcommand{\rH}[2][\star]	{{#2}^{#1}{}}				% right Hodge dual
\newcommand{\lHB}[1]		{\lH[\star_B]{#1}}			% left Hodge dual wrt B
\newcommand{\rHB}[1]		{\rH[\star_B]{#1}}			% right Hodge dual wrt B
\newcommand\hhat[2]			{#1\string^ {}^{#2}} % {#1 {\hat{\,}}^{\,#2}}		% multiple grade involution
\newcommand{\grade}[1]		{|#1|}						% grade
\newcommand{\tgrade}[1]		{|#1|_{\text{top}}}			% top grade
\newcommand{\bgrade}[1]		{|#1|_{\text{bot}}}			% bottom grade
\newcommand{\igrade}[1]		{|#1|_{\text{in}}}			% inner grade
\newcommand{\ograde}[1]		{|#1|_{\text{out}}}			% outer grade
\newcommand\comp[2] 		{(#1)_{#2}}		% homog component
\def\pperp					{\simperp}					% partially othogonal 
\def\ext					{e}						% exterior prod operator

\DeclareMathOperator{\Span}{span}

\DeclareMathOperator{\Ann}{Ann}

\DeclareMathOperator{\End}{End}
\DeclareMathOperator{\Img}{Im}
\DeclareMathOperator{\codim}{codim}

\begin{document}

\title{Multivector Contractions Revisited, Part II}

\author{Andr\'e L. G. Mandolesi 
               \thanks{Instituto de Matemática e Estatística, Universidade Federal da Bahia, Av. Adhemar de Barros s/n, 40170-110, Salvador - BA, Brazil. E-mail: \texttt{andre.mandolesi@ufba.br}}}
               
\date{\today \SELF{v3.0 Separada}}

\maketitle

\abstract{
The theory of contractions of multivectors, and star duality, was reorganized in a previous article, and here we present some applications. 
First, we study inner and outer spaces associated to a general multivector $M$ via the equations $v \wedge M = 0$ and $v \lcontr M=0$.
They are then used to analyze special decompositions, factorizations and `carvings' of $M$,
to define generalized grades, 
and to obtain new simplicity criteria, including a reduced set of Plücker-like relations.
We also discuss how contractions are related to supersymmetry, and give formulas for supercommutators of multi-fermion creation and annihilation operators.

\vspace{.5em}
\noindent
{\bf Keywords:} Grassmann exterior algebra, Clifford geometric algebra, contraction, inner and outer space, simplicity, Plücker, supersymmetry.

\vspace{3pt}

\noindent
{\bf MSC:} 	15A75,	%Exterior algebra, Grassmann algebras
			15A66,  %Clifford algebras, spinors
			81Q60	%Supersymmetry and quantum mechanics
%			14M15	%Grassmannians, Schubert varieties, flag manifolds
}

\section{Introduction}

In \cite{Mandolesi_Contractions}, we reorganized and extended the theory of contractions or interior products of multivectors, and Hodge star duality. 
In this article, we discuss some applications, organizing and generalizing some known results, and obtaining new ones.
\Cref{sc:Spaces of Non-Simple Multivectors} studies inner and outer spaces \cite{Kozlov2000I,Rosen2019} associated to a general multivector $M$ via the solution sets of equations $v \wedge M = 0$ and $v \lcontr M =0$.
They are related to the outer and inner product null spaces of blades (simple multivectors) in Clifford Geometric Algebra \cite{BayroCorrochano2018},
	\CITE{Hitzer2012}
and provide information about the structure of $M$.
\Cref{sc:Special operations} uses them to analyze special ways in which $M$ can be written as a sum, exterior product or contraction of other elements.
\Cref{sc:Simplicity and Multivector Grades} defines generalized grades and gives simplicity criteria, conditions for a multivector to be a blade.
These include some little known ones from \cite{Eastwood2000}, and others which are new, like a reduced set of equations similar to the Plücker relations \cite{Gallier2020,Jacobson1996} used to describe Grassmannians \cite{Kozlov2000I}.
\Cref{sc:Supersymmetry} relates contractions to  supersymmetry \cite{Oziewicz1986,Shaw1983,Varadarajan2004}, and gives new formulas for supercommutators of multi-fermion creation and annihilation operators.

As before, $X$ is an $n$-dimensional Euclidean or Hermitian space, with inner or Hermitian product $\inner{\cdot,\cdot}$, and $\bigwedge X$ is its exterior algebra.
In general, we use $L,M,N$ for arbitrary multivectors, $F,G,H$ for homogeneous ones, and $A,B,C$ for simple ones (blades).
We use the results and notation of \cite{Mandolesi_Contractions} (see \Cref{tab:symbols}), and include some below for easy reference.

\begin{table}[]
	\scriptsize
	\centering
	\renewcommand{\arraystretch}{1}
	\begin{tabular}{ll}
		\toprule
		Symbol & Description %& Uso
		\\
		\cmidrule(lr){1-1} \cmidrule(lr){2-2} % \cmidrule(lr){3-3}
		%
%		$X$ & $n$-dimensional Euclidean or Hermitian space & % \pageref{df:X}
%		\\
%		$\bigwedge V$, $\bigwedge^p V$, $\bigwedge^+ V$ & Exterior algebra, exterior power, even subalgebra & % \pageref{df:exterior algebra} % (Sec. \ref{sc:Grassmann Algebra})
%		\\
%		$\inner{\cdot,\cdot}$ & Inner or Hermitian product in $X$ or $\bigwedge X$ & % \pageref{df:inner}, \pageref{df:inner AB} 
%		\\
%		$\wedge$, $\vee$ & Exterior and regressive products & % \pageref{df:wedge}, \pageref{df:regressive}
%		\\
		$\lcontr$, $\rcontr$ & Left and right contractions %& p2
		\\
%		$\isp{M}$, $\osp{M}$,  & Inner and outer spaces of $M$ & \pageref{df:isp osp} %(Def. \ref{df:isp osp})
%		\\
		$\lH{\! M}$, $\rH{M}$ & Left and right duals, $\lH{\! M} = \Omega \rcontr M$, $\rH{M} = M \lcontr \Omega$, for a unit $\Omega \in \bigwedge^n X$ %& p2
		\\
		$[B]$ & $\Span\{v_1,\ldots,v_p\}$ for blade $B = v_1\wedge\cdots\wedge v_p \neq 0$; $[\lambda]=\{0\}$ for scalar $\lambda$ %& p2
		\\
		$\grade{H}$ & Grade of a homogeneous multivector $H$  
		\\
		$\comp{M}{p}$ & Component of grade $p$ of $M$ %& p2,6
		\\
		$\hat{M}$ & Grade involution, $\hat{M} = \sum_p (-1)^p \comp{M}{p}$  %& p.2,5
		\\
		$\hhat{M}{k}$ & Composition of $k$ grade involutions %& p.2,14
		\\
		$\tilde{M}$ & Reversion, $\tilde{M} = \sum_p (-1)^\frac{p(p-1)}{2} \comp{M}{p}$  
		\\
		$P_V$, $P_B$, $\PP_\VV$ & Orthogonal projections $P_V:X \rightarrow V$, $P_B=P_{[B]}$, $\PP_\VV:\bigwedge X \rightarrow \VV$ %& p2
		\\
		$\II^q_p$ & $\{ (i_1,\ldots,i_p)\in\N^p : 1\leq i_1 < \cdots<i_p\leq q \}$, and $\II^q_0 = \{\emptyset\}$ %& p2,12
		\\
		$\MM^q_p$ & $\{ (i_1,\ldots,i_p)\in\N^p : 1\leq i_j \leq q, i_j\neq i_k \text{ if } j\neq k \}$, and $\MM^q_0 = \{\emptyset\}$ %& p2,13
		\\
		$\II^q$, $\MM^q$ & $\II^q = \bigcup_{p=0}^q \II_p^q$ and $\MM^q = \bigcup_{p=0}^q \MM_p^q$
		\\
		$v_\rr$, $v_{i_1\cdots i_p}$ & $v_{i_1}\wedge\cdots\wedge v_{i_p}$ for $v_1,\ldots,v_q \in X$ and $\rr = (i_1,\ldots,i_p)$, and $v_\emptyset = 1$ %& p2
		\\
		$|\rr|$ & $|\rr|=p$ for $\rr = (i_1,\ldots,i_p)$
		\\
		$\rr\ss$ & Unordered concatenation of $\rr,\ss\in \MM^q$ 
		\\
		$\ii \cup \jj$, $\ii \cap \jj$,$\ii \backslash \jj$,  $\ii \triangle \jj$ & Ordered union, intersection, difference, symmetric difference of $\ii,\jj\in\II^q$
		\\
		$\ii'$ & $(1,\ldots,q) \backslash \ii$ for $\ii\in\II^q$
		\\
		$\epsilon_{\rr}$ & Sign of the permutation that puts $\rr \in \MM^q$ in increasing order %& p2,13
		\\
		$\ord{\rr}$ & Reordering of $\rr \in \MM^q$ so that $\ord{\rr}\in\II^q$ %& p2,13
		\\
		$\pairs{\rr}{\ss}$ & Number of pairs $(i,j)$ with $i\in\rr$, $j\in\ss$, $i>j$ %& p2,13	
		\\
		$\pperp$ & Partial orthogonality, $U \pperp V \Leftrightarrow V^\perp \cap U \neq \{0\}$ %& p5
		\\
%		$\igrade{M}$, $\ograde{M}$, $\bgrade{M}$, $\tgrade{M}$ & Inner, outer, bottom and top grades of $M$ & \pageref{df:bt grades} %(Def. \ref{df:bt grades})
%		\\
%		$\adj{T}$ & Adjoint of a linear map $T$ 
%		\\
%		$[S,T]$, $\{S,T\}$, $\scom{S,T}$ & Commutator, anticommutator, supercommutator & \pageref{df:supercommutator} % (Sec. \ref{sc:Supersymmetry})
%		\\
%		$\ext_M, \iota_M$ & Exterior and interior products, $\ext_M(N) = M \wedge N$, $\iota_M(N) = M \lcontr N$ & p14 %\pageref{df:ext int products} %(Def. \ref{df:ext int products}) 
%		\\
%		$a_\rr$, $a_\rr^\dagger$ & Annihilation and creation operators & \pageref{df:annihilation creation} %(Def. \ref{df:annihilation creation})
%		\\
		$\delta_{\mathbf{P}}$ & Propositional delta, 1 if $\mathbf{P}$ is true, 0 otherwise  %& p15
		\\
		$M\wedge \bigwedge V$ & $\{M\wedge N:N\in\bigwedge V\}$  %& p15
		\\
%		$\VV_\ii$, $\WW_\ii$ & Subspaces of $\bigwedge X$ related to $a_\ii$, $a_\ii^\dagger$ & \pageref{df:Vi Wi} % (Def. \ref{df:Vi Wi})
%		\\
%		$m_\ii$, $n_\ii$ & Vacancy and occupancy operators & \pageref{df:vacancy occupancy} %(Def. \ref{df:vacancy occupancy})
%		\\
		\bottomrule
	\end{tabular}
	\caption{Some symbols defined in \cite{Mandolesi_Contractions}.}
	\label{tab:symbols}
\end{table}

\begin{proposition}\label{pr:Hodge wedge contr}
	$\rH{(M\wedge N)} = N\lcontr\rH{M}$ and $\rH{(M\rcontr N)} = N\wedge\rH{M}$, for $M,N\in\bigwedge X$.%
	\OMIT{p.\pageref{uso pr:Hodge wedge contr}}
\end{proposition}

\begin{proposition}\label{pr:contr homog}
	$v\lcontr(M\wedge N) = (v\lcontr M)\wedge N + \hat{M}\wedge(v\lcontr N)$, and $v \wedge (M \lcontr N) = (M \rcontr v) \lcontr N + \hat{M} \lcontr (v \wedge N)$, for $v \in X$ and $M,N\in\bigwedge X$.
	\OMIT{p.\pageref{uso pr:contr homog}, p.\pageref{uso2 pr:contr homog}}
\end{proposition}

\begin{proposition}\label{pr: v contr M wedge N = 0}
	For $v\in X$ and nonzero $M\in\bigwedge V$ and $N\in\bigwedge W$ with $V\cap W = \{0\}$, $v\lcontr(M\wedge N) = 0  \Leftrightarrow v\lcontr(MN) = 0 \Leftrightarrow v\lcontr M = v\lcontr N = 0$.
		\OMIT{p.\pageref{uso pr: v contr M wedge N = 0}}
\end{proposition}

\begin{proposition}\label{pr:wedge contr ker image}
	For $0 \neq v \in X$ and $M \in \bigwedge X$:
		\OMIT{p.\pageref{uso pr:wedge contr ker image}}
	\begin{enumerate}[i)]
		\item $v\wedge M = 0 \Leftrightarrow M = v\wedge N$ for $N\in\bigwedge X$. In particular, $N = \frac{v \lcontr M}{\|v\|^2}$.\label{it:wedge ker image}
		\item $v\lcontr M = 0 \Leftrightarrow M = v\lcontr N$ for $N\in\bigwedge X$. In particular, $N = \frac{v\wedge M}{\|v\|^2}$.\label{it:contr ker image}
	\end{enumerate}
\end{proposition}

\begin{proposition}\label{pr:M=BBM}
	Let $B \neq 0$ be a blade, and $M\in\bigwedge X$.
		\OMIT{p.\pageref{uso pr:M=BBM}, p.\pageref{uso2 pr:M=BBM}}
	\begin{enumerate}[i)]
		\item $M = N \lcontr B$ for $N \in \bigwedge X \Leftrightarrow M = L \lcontr B$ for $L = \frac{B\rcontr M}{\|B\|^2} \in \bigwedge([B])$. 
		\item $M = B \wedge N$ for $N \in \bigwedge X \Leftrightarrow M = B \wedge L$ for $L = \frac{B \lcontr M}{\|B\|^2} \in \bigwedge([B]^\perp)$.
%		\item $M = B \lcontr N$ for $N \in \bigwedge X \Leftrightarrow M = B\lcontr L$ for $L = \frac{B\wedge M}{\|B\|^2} \in B \wedge \bigwedge([B]^\perp)$.
	\end{enumerate}
\end{proposition}

\begin{proposition}\label{pr:triple subblade}
	Let $A$ and $B$ be blades, and $L,M \in\bigwedge X$.\OMIT{p.\pageref{uso pr:triple subblade}}
	\begin{enumerate}[i)]
		\item $(A\rcontr M)\lcontr B = (P_A M)\wedge(A\lcontr B)$, if $[A] \subset [B]$. \label{it:AMB}
		\item $(M\rcontr L)\lcontr B = L\wedge(M\lcontr B)$, if $L \in \bigwedge [B]$. \label{it:MAB}
	\end{enumerate}
\end{proposition}

\begin{proposition}\label{pr:epsilon}
	$\epsilon_{\ii\jj} = (-1)^{\pairs{\ii}{\jj}}$ and $\epsilon_{\ii\jj\kk} = \epsilon_{\ii \jj} \, \epsilon_{\ii\kk} \, \epsilon_{\jj\kk}$,
	for pairwise disjoint $\ii,\jj,\kk \in\II^q$.\OMIT{p.\pageref{uso pr:epsilon}}
	%	\begin{enumerate}[i)]
		%		\item $\epsilon_{\rr \ss} = (-1)^{|\rr||\ss|} \epsilon_{\ss \rr}$. \label{it:epsilon graded commutation adjacent}		
		%		\item $\epsilon_{\rr \ss} = \epsilon_{\rr} \, \epsilon_{\ord{\rr}\ss}$. 
		%		\label{it:epsilon ord ra}
		%	\end{enumerate}
\end{proposition}

\begin{proposition}\label{pr:higher order}
	Let $v_1,\ldots,v_p\in X$ and $M,N\in\bigwedge X$.\OMIT{p.\pageref{uso pr:higher order}}
	\begin{enumerate}[i)]
		\item $v_{1\cdots p}\lcontr(M\wedge N) = \sum_{\ii\in\II^p} \epsilon_{\ii'\ii} (\hhat{M}{p} \rcontr \hat{v}_{\ii}) \wedge (v_{\ii'} \lcontr N)$. \label{it:generalized Leibniz}
		\item $v_{1\cdots p} \wedge (M \lcontr N) = \sum_{\ii\in\II^p} \epsilon_{\ii\ii'} (\hat{v}_{\ii} \lcontr \hhat{M}{p}) \lcontr (v_{\ii'} \wedge N)$. \label{it:v1p wedge M contr N}
		%
%		\item $M \lcontr (v_{1\cdots p} \wedge N) = \sum_{\ii\in\II^p} \epsilon_{\ii'\ii}\, v_{\ii'}\wedge \bigl((\hhat{M}{p} \rcontr \hat{v}_{\ii}) \lcontr N)\bigr)$. \label{it:generalized Leibniz adjoint}
%		%
%		\item $M \wedge (v_{1\cdots p} \lcontr N) = \sum_{\ii\in\II^p} \epsilon_{\ii\ii'}\, v_{\ii'} \lcontr \bigl((\hat{v}_{\ii} \lcontr \hhat{M}{p}) \wedge N)\bigr)$. \label{it:v1p wedge M contr N adjoint}
	\end{enumerate}
\end{proposition}

These are slightly different from the versions in \cite{Mandolesi_Contractions}: 
e.g., in \ref{it:generalized Leibniz} we used $v_{\ii'}\lcontr \hhat{M}{|\ii|} = \hhat{M}{p} \rcontr \hat{v}_{\ii'}$ and relabeled $\ii \leftrightarrow \ii'$, for later convenience.
% $\hhat{M}{|\ii'|} \rcontr v_{\ii} = \hat{v}_{\ii} \lcontr \hhat{M}{p}$

\begin{proposition}\label{pr:triple}
	Let $\ext_B(M)=B\wedge M$ and $\iota_B(M)=B\lcontr M$,
	for a unit blade $B$ and $M\in\bigwedge X$.
	Then $\Img \ext_B = B \wedge \bigwedge([B]^\perp)$ and
	$\Img \iota_B = \bigwedge ([B]^\perp)$.
	Also,
	$\iota_B(\ext_B(M)) =\PP_{\Img \iota_B} M$,
	and $\ext_B(\iota_B(M)) = \PP_{\Img \ext_B} M$.
	\OMIT{p.\pageref{uso pr:triple}}
\end{proposition}

\section{Inner and outer spaces}\label{sc:Spaces of Non-Simple Multivectors}

A blade $B=v_1\wedge\cdots\wedge v_p \neq 0$ is often used to represent a subspace $[B] = \Span\{v_1,\ldots,v_p\} = \{v\in X:v\wedge B = 0\} = \{v\in X:v \lcontr B = 0\}^\perp$.
For a general multivector $M \in \bigwedge X$, this concept splits into two spaces, which, as we show, give important information about $M$.
	\CITE{Kozlov2000I,Rosen2019 have homogeneous case of some results}

\begin{definition}\label{df:isp osp}
	The \emph{inner} and \emph{outer spaces} of $M  \in \bigwedge X$ are, respectively, $\isp{M} = \{v\in X:v\wedge M=0\}$ and $\osp{M} = \{v\in X: v\lcontr M = 0\}^\perp$.
\end{definition}
	\SELF{$\isp{\hat{M}} = \isp{M}$, etc}

The notation and terminology is from Rosén \cite[p.\,36]{Rosen2019}, who does not define the spaces for ``inhomogeneous multivectors, since this is not natural geometrically''.
As we show, the theory works just as well with them.
For a homogeneous $H$, $\isp{H}$ is in \cite{Kozlov2000I} the \emph{annihilator} $\Ann H$,
and $\osp{H}$ is its \emph{rank space}.%
	\OMIT{\ref{pr:isp osp homog}}% 
	\CITE{Shaw1983 p.365 defines rank as its dim}
In Clifford Geometric Algebra \cite{BayroCorrochano2018}, 
	\CITE{Hitzer2012}
the \emph{outer product null space} of a blade $B$ is $\isp{B}$, and its \emph{inner product null space} is $\osp{B}^\perp$.

\begin{proposition}\label{pr:blade osp=isp=space}
	$\isp{B}=\osp{B}=[B]$, for a blade $B\neq 0$. 
\end{proposition}
\begin{proof}
	Immediate for $B$ as above. If $B$ is scalar, all spaces are $\{0\}$.
%	$\osp{B}=[B]$ as $v \lcontr B = 0 \Leftrightarrow v \in [B]^\perp$.
%	\OMIT{pr:contr v perp M}
%	If $B$ is scalar,  $\isp{B} = [B] = \{0\}$.	
%	If $B=v_1\wedge\cdots\wedge v_p$,
%	$v \in [B] = \Span\{v_1,\ldots,v_p\} \Leftrightarrow v \wedge B = 0$.
	%	$[B]=\Span\{v_1,\ldots,v_p\} \subset \isp{B}$ as $v_i\wedge B=0 \ \forall i$, and if $v\in\isp{B}$ then $v\wedge v_1\wedge\cdots\wedge v_p = 0$, so $v \in \Span\{v_1,\ldots,v_p\}$. 
\end{proof}

Zero is an important exception: $\osp{0}=[0]=\{0\}$, but $\isp{0}=X$.
%As we show later (\Cref{pr:isp=osp iff simple}), inner and outer spaces coincide only for blades.

\begin{proposition}\label{pr:isp subset osp}
	$\isp{M}\subset \osp{M}$, for $0 \neq M \in \bigwedge X$.
\end{proposition}
\begin{proof}
	$v\in \isp{M}$, $w\in \osp{M}^\perp
%	$v\wedge M = w \lcontr M = 0 
	\Rightarrow 0 = w \lcontr (v \wedge M) =	\inner{w,v} M \Rightarrow v \perp w$.
		\OMIT{\ref{pr:contr homog}}\label{uso pr:contr homog}
\end{proof}
%	\OMIT{\ref{pr:blade osp=isp=space}, \ref{pr:inter osp 0 wedge not 0}}
%	$0\neq v\in\isp{M} \Rightarrow v\wedge M = 0 \Rightarrow [v] \cap \osp{M} \neq \{0\} \Rightarrow v \in \osp{M}$.
	%	, as $\osp{v}=\Span\{v\}$.
	%	Let $0\neq v\in\isp{M}$. As $v\wedge M = 0$, we must have $\osp{v}\cap \osp{M} \neq \{0\}$.
	%	As $\osp{v}=\Span\{v\}$, this means $v \in \osp{M}$.

%So, for $M\neq 0$, solutions of $v\wedge M = 0$ and $v\lcontr M = 0$ are orthogonal.

The relation $[\rH{B}] = [B]^\perp$, for a blade $B\neq 0$, splits into the following ones (likewise for $\lH{M}$).
They hold for $M=0$ due to the above exception.

\begin{proposition}\label{pr:osp isp star}
	$\isp{\rH{M}} = \osp{M}^\perp$ and $\osp{\rH{M}} = \isp{M}^\perp$, for $M \in \bigwedge X$.
\end{proposition}
\begin{proof}
	By \Cref{pr:Hodge wedge contr},\label{uso pr:Hodge wedge contr}
	$0 = v\wedge \rH{M} = \rH{(M\rcontr v)} \Leftrightarrow v \lcontr M = 0$, and
	$0 = v\lcontr \rH{M} = \rH{(M\wedge v)} \Leftrightarrow v \wedge M = 0$.
\end{proof}

Next we give a few alternative characterizations of these spaces.

\begin{proposition}\label{pr:isp osp homog}
	Let $H\in\bigwedge^p X$.
	\begin{enumerate}[i)]
		\item $\osp{H} = \{F\lcontr H: F \in \bigwedge^{p-1} X\}$. \label{it:osp lcontr p-1}
		\SELF{$= (\bigwedge^{p-1} X)\lcontr H$}
		\item $\isp{H} = \{H\lcontr G: G \in \bigwedge^{p+1} X\}^\perp$. \label{it:isp lcontr p+1}
		\SELF{$= (H \lcontr (\bigwedge^{p+1} X))^\perp$}
	\end{enumerate}
\end{proposition}
\begin{proof}
	\emph{(\ref{it:osp lcontr p-1})} 
	$v \in \osp{H}^\perp \Leftrightarrow \inner{v,F\lcontr H} = \inner{F\wedge v, H} = \inner{F,H\rcontr v} = 0 \ \forall F\in \bigwedge^{p-1} X$.
	%	 \Leftrightarrow v \in \{F\lcontr H: F \in \bigwedge^{p-1} X\}^\perp$.
	%
	\emph{(\ref{it:isp lcontr p+1})} 
	$v \in \isp{H} \Leftrightarrow \inner{v,H \lcontr G} = \inner{H \wedge v, G} = 0 \ \forall G \in \bigwedge^{p+1} X$.
	%		\OMIT{\ref{it:isp lcontr p+1}, it:star wedge contr}
	%	$\isp{H}^\perp = \osp{\lH{H}} = \{F\lcontr \lH{H}: F \in \bigwedge^{n-p-1} X\}$
	%	and $F\lcontr \lH{H} = \rH{F} \rcontr H = H\lcontr G$ for $G=\pm \rH{F}$.
\end{proof}

\begin{proposition}\label{pr:osp smallest subspace with M}
	$\osp{M} = \bigcap\{V\subset X:  M\in\bigwedge V\}$ is the smallest subspace whose exterior algebra contains $M \in \bigwedge X$. 
\end{proposition}
\begin{proof}
	As $v\lcontr M = 0 \Leftrightarrow M \in \bigwedge([v]^\perp)$,%
		\OMIT{pr:contr v perp M} 
	a basis $(v_1,\ldots,v_k)$ of $\osp{M}^\perp$
	gives $M \in \bigcap_{i=1}^k \bigwedge([v_i]^\perp) =  \bigwedge \bigcap_{i=1}^k ([v_i]^\perp) = \bigwedge \osp{M}$.
		\OMIT{$= \bigwedge \big( (\sum_{i=1}^k [v_i])^\perp \big)$}
	And $M \in \bigwedge V \Rightarrow \osp{M} \subset V$, as $v \in V^\perp \Rightarrow V \subset [v]^\perp \Rightarrow M \in \bigwedge([v]^\perp) 
	\Rightarrow v\lcontr M = 0 
	\Rightarrow v \in \osp{M}^\perp$.
\end{proof}

Dualizing, $\isp{M} = \sum\{V\subset X: \rH{M}\in\bigwedge (V^\perp)\}$ is the largest subspace whose orthogonal complement has $\rH{M}$ in its exterior algebra. 
	\SELF{$= \left(\bigcap\{V\subset X:  \rH{M}\in\bigwedge V\}\right)^\perp$, orth compl of smallest $V$ with $\rH{M}$}

%\begin{proposition}\label{pr:M in osp}
%	$M\in \bigwedge \osp{M}$.
%\end{proposition}
%\begin{proof}
%	Given a basis $(v_1,\ldots,v_k)$ of $\osp{M}^\perp$,
%	as $v_i\lcontr M = 0 \Leftrightarrow M \in \bigwedge([v_i]^\perp)$%
%		\OMIT{pr:contr v perp M} 
%	we have $M \in \bigcap_{i=1}^k \bigwedge([v_i]^\perp) =  \bigwedge \bigcap_{i=1}^k ([v_i]^\perp) = \bigwedge \osp{M}$.
%		\OMIT{$= \bigwedge \big( (\sum_{i=1}^k [v_i])^\perp \big)$}
%\end{proof}

As $M \in \bigwedge \osp{M}$, the following just restates a well known result.

\begin{corollary}\label{pr:inter osp 0 wedge not 0}
	Given $M,N \in \bigwedge X$ with $\osp{M}\cap\osp{N}=\{0\}$, we have
	$M\wedge N = 0 \Leftrightarrow M=0$ or $N=0$. 
		\OMIT{\ref{pr:osp smallest subspace with M} \\ Usual result that, for $M\in\bigwedge V$ and $N\in\bigwedge W$ with $V\cap W=\{0\}$, $M\wedge N=0 \Leftrightarrow M=0$ or $N=0$. \\
		Detalhando: bases de $V$ e $W$ formam de $V\oplus W$, induzindo base de $\bigwedge(V\oplus W)$ em termos de bases de $\bigwedge V$ e $\bigwedge W$. Decompor $M$, $N$ e $M\wedge N$ nessas bases.}
\end{corollary}
%\begin{proof}
%	As $M \in \bigwedge \osp{M}$ and $N \in \bigwedge \osp{N}$, this is a well known result.
%%	 for $M\in \bigwedge V$, $N \in \bigwedge W$ with $V\cap W = \{0\}$.
%\end{proof}

\begin{proposition}\label{pr:isp alternative}
	Let $M\in \bigwedge X$.
	\begin{enumerate}[i)]
		\item $\isp{M} = \{v\in X: v=0$ or $M = v\wedge N$ for $N\in\bigwedge X\}$. \label{it:isp factor v}
		
		\item $\isp{M} = [B]$ for any blade $B \neq 0$ of largest possible grade such that $M = B\wedge N$ for $N\in\bigwedge X$. \label{it:isp space max factor} 
	\end{enumerate}
\end{proposition}
\begin{proof}
	\emph{(\ref{it:isp factor v})} Follows from \Cref{pr:wedge contr ker image}.\label{uso pr:wedge contr ker image}
%	Given $v \neq 0$ and a complement $V$
%		\SELF{Poderia pegar $V=[v]^\perp$, mas deixo assim para ver que não é propriedade métrica}
%	of $[v]$, $M = v\wedge N + L$ for $N,L \in \bigwedge V$, so $v\in\isp{M} \Rightarrow 0 = v \wedge M = v \wedge L \Rightarrow L=0$.
%	The converse is immediate.
%	Conversely, $M = v\wedge N \Rightarrow v\wedge M =0 \Rightarrow v\in\isp{M}$.
	%
	\emph{(\ref{it:isp space max factor})} 
	Let $M = B\wedge N$ and assume $N \in \bigwedge ([B]^\perp)$, by \Cref{pr:M=BBM}.
	Clearly, $[B] \subset \isp{M}$,
		\OMITL{$v\in [B] \Rightarrow v\wedge M = v \wedge B \wedge N = 0 \Rightarrow v\in \isp{M}$}
	and if there is $0\neq v \in [B]^\perp \cap \isp{M}$ then
	$0 = v\wedge M = \hat{B} \wedge (v \wedge N) \Rightarrow v \wedge N =0 \Rightarrow N=v \wedge L$
		\OMIT{\ref{pr:inter osp 0 wedge not 0}, $v \wedge N \in \bigwedge ([B]^\perp)$}
	for $L\in \bigwedge X$, by \ref{it:isp factor v}, so
	$M=(B\wedge v) \wedge L$ and $B$ is not maximal.
\end{proof}

As $0 = B\wedge 0$ for any $B$, this agrees with $\isp{0}=X$.
%In \Cref{sc:Blade Factorizations} we show $\isp{M}$ determines which blades can be factored out of $M$, and give conditions that ensure the uniqueness of $N$.

\begin{corollary}\label{pr:isp=osp blade}
	If $\isp{M} = \osp{M}$ then $M\in \bigwedge X$ is a blade.
\end{corollary}
\begin{proof}\OMIT{\ref{pr:osp smallest subspace with M}, \ref{pr:isp alternative}\ref{it:isp space max factor} }
	Let $M=B\wedge N$ with $[B]=\isp{M}$ and $N \in \bigwedge ([B]^\perp)$.
		\OMIT{\ref{pr:M=BBM}}
	If $\isp{M} = \osp{M}$ then $B\wedge N \in \bigwedge \osp{M} = \bigwedge [B]$, and so $N$ is a scalar.
		\OMIT{expand $N$ in basis of $V$}
\end{proof}

\begin{corollary}\label{pr:inter isp wedge 0}
	$M\wedge N \neq 0 \Rightarrow \isp{M}\cap\isp{N} = \{0\}$, for $M,N\in\bigwedge X$.
	\OMIT{\ref{pr:isp alternative}\ref{it:isp factor v}: $M = v\wedge M'$ and $N = v\wedge N'$ if there is $0\neq v\in \isp{M}\cap\isp{N}$}
	%	$\isp{M}\cap\isp{N}\neq\{0\} \Rightarrow M\wedge N=0$, for $M,N\in\bigwedge X$.
\end{corollary}

The next example shows it does not imply $\osp{M}\cap\osp{N} = \{0\}$. 

\begin{example}\label{ex:wedge e inter non 0}
	If $N = 1 + v \wedge w$ for a basis $(v,w)$ of $\R^2$ then $v\wedge N = v \neq 0$ and $\isp{v} \cap \isp{N} = [v] \cap \{0\} =\{0\}$, but $\osp{v} \cap \osp{N} = [v] \cap \R^2 \neq \{0\}$.
\end{example}

\begin{proposition}\label{pr:osp alternative}
	Let $M\in \bigwedge X$.
	\begin{enumerate}[i)]
		\item $\osp{M} = \{v\in X: v=0$ or $M = v\lcontr N$ for $N\in\bigwedge X\}^\perp$. \label{it:osp v carv}
		
		\item $\osp{M} = [B]$ for any blade $B$ of smallest possible grade such that $M = N \lcontr B$ for $N\in\bigwedge X$. \label{it:osp space max carv}
	\end{enumerate}
\end{proposition}
\begin{proof}
	\emph{(\ref{it:osp v carv})} Follows from \Cref{pr:wedge contr ker image}.
%	$\osp{M}^\perp = \isp{\lH{M}} = \{v\in X: v=0$ or $\lH{M} = L\wedge v$ for $L\in\bigwedge X\}$, and $\lH{M} = L\wedge v \Leftrightarrow M = \rH{(L\wedge v)} = v \lcontr N$ for $N = \rH{L}$.
%		\OMIT{\ref{pr:isp alternative}\ref{it:isp factor v}, it:star inverses, it:star wedge contr}
	%
	\emph{(\ref{it:osp space max carv})} 	
	If $A$ is a maximal blade such that $\lH{M} = A \wedge N$ for $N\in \bigwedge X$,
	then
	$\osp{M} = \isp{\lH{M}}^\perp = [A]^\perp = [\rH{A}]$,
	and $\lH{M} = A \wedge N \Leftrightarrow M = \rH{(A \wedge N)} = N \lcontr \rH{A}$.%
		\OMIT{\ref{pr:osp isp star}, \ref{pr:isp alternative}\ref{it:isp space max factor}, \ref{pr:Hodge wedge contr}, it:star inverses, it:left right star}
\end{proof}

%\begin{corollary}\label{pr:wedge contr ker image}
%	Let $0\neq v\in X$ and $M\in\bigwedge X$.
%	\begin{enumerate}[i)]
%		\item $v\wedge M = 0 \Leftrightarrow M = v\wedge N$ for $N\in\bigwedge X$. In particular, $N = \frac{v \lcontr M}{\|v\|^2}$.\label{it:wedge ker image}
%		\item $v\lcontr M = 0 \Leftrightarrow M = v\lcontr N$ for $N\in\bigwedge X$. In particular, $N = \frac{v\wedge M}{\|v\|^2}$.\SELF{More generally, for $N = \frac{v+L}{\|v\|^2} \wedge M$ with $L \in \bigwedge ([v]^\perp)$}\label{it:contr ker image}
%	\end{enumerate}
%\end{corollary}

\begin{example}\label{ex:spaces}
	Let $M = v_{134} - v_{145} + v_{345} + v_{1235}$ for an orthonormal basis $(v_1,\ldots,v_5)$.
	Solving $v\wedge M=0$ gives $\isp{M} = \Span\{v_1-v_3,v_3+v_5\}$,
	\SELF{ou $\Span\{v_1+v_5,v_3+v_5\}$}
	so we can factorize $M = (v_1-v_3)\wedge (v_3+v_5) \wedge (v_4-v_{25})$, for example.
	Solving $v\lcontr M=0$ we only find $v=0$, so $\osp{M} = X$.
\end{example}

\begin{example}\label{ex:spaces 2}
	Let $M = v_{123} + 2 v_{145} - v_{146}$ for an orthonormal basis $(v_1,\ldots,v_6)$.
%	Solving $v\wedge M=0$, w
	As $\isp{M} = [v_1]$, only multiples of $v_1$ can be factored out.
%	Solving $v\lcontr M=0$, we obtain 
	And
	$\osp{M} = [v_5 + 2v_6]^\perp = [\rH{v_5} + 2\rH{v_6}] = [v_{12346}-2v_{12345}] = [v_{1234}\wedge (v_6-2v_5)]$,
	so, for example, 
	$M = \left(v_{46} - v_{23}\right) \lcontr \left(v_{1234}\wedge (v_6-2v_5)\right)$.%
		\SELF{and $M = (v_5 + 2v_6) \lcontr \frac{(v_5 + 2v_6)\wedge (v_{123} + 2 v_{145} - v_{146})}{5}$}
\end{example}

Now we describe how inner and outer spaces behave under outermorphisms, exterior products, contractions and Clifford products.

\begin{proposition}\label{pr:T osp isp}
	Let $T:X\rightarrow Y$ be a linear map into a Euclidean or Hermitian space $Y$ (same as $X$),
	and $M\in\bigwedge X$.
%	and $T:\bigwedge X \rightarrow \bigwedge Y$ be its outermorphism.
	Then $\isp{T M} \supset T(\isp{M})$ and $\osp{TM} \subset T(\osp{M})$, with equalities if $T$ is invertible.
\end{proposition}
\begin{proof}
	For $v\in\isp{M}$, $Tv \wedge TM = T(v\wedge M) = 0$, so $Tv\in\isp{TM}$.
	And $TM \in T(\bigwedge\osp{M}) = \bigwedge T(\osp{M})$.
	\OMIT{\ref{pr:osp smallest subspace with M}}
	If $T$ is invertible, $\isp{M} = \isp{T^{-1}TM} \supset T^{-1}(\isp{TM})$, so $\isp{T M} \subset T(\isp{M})$, 
	and likewise $\osp{TM} \supset T(\osp{M})$.
\end{proof}

\begin{proposition}\label{pr:osp isp wedge}
	For $M,N \in \bigwedge X$ we have the following, with equalities and $\oplus$ if $M,N\neq 0$ and $\osp{M}\cap\osp{N}=\{0\}$.%
		\SELF{If $M=N=0$ we still have $\osp{M\wedge N} = \osp{M}\oplus\osp{N} = \{0\}$. If $M=0$ or $N=0$ we have $\isp{M\wedge N} = \isp{M}+\isp{N} = X$.}%
		\CITE{Rosen proves for homogeneous, and equalities for blades}%
	\begin{enumerate}[i)]
		\item $\osp{M\wedge N} \subset \osp{M}+\osp{N}$. \label{it:osp wedge}
		\item $\isp{M\wedge N} \supset \isp{M}+\isp{N}$. \label{it:isp wedge}
			\MAYBE{poderia ser mais fácil fazer primeiro o \ref{pr:osp isp lcontr}\ref{it:osp lcontr} e então dualizar para obter este?}
	\end{enumerate}
\end{proposition}
\begin{proof}
	\emph{(\ref{it:osp wedge})}
		\OMIT{\ref{pr:osp smallest subspace with M}}
	$M \in \bigwedge \osp{M}$, $N \in \bigwedge\osp{N} \Rightarrow M\wedge N \in \bigwedge(\osp{M}+\osp{N})$.
		\OMITL{$M\wedge N \in \bigwedge(\osp{M}+\osp{N})$}
	\emph{(\ref{it:isp wedge})} $v\in\isp{M} + \isp{N} \Rightarrow v=u+w$ with $u\wedge M = w \wedge N = 0$, so $v\wedge M\wedge N = 0$.
	%	$v\in \isp{M\wedge N}$.	
		
	Now let $M,N\neq 0$ and $\osp{M}\cap\osp{N}=\{0\}$.
	\emph{(\ref{it:osp wedge})} 
	By \Cref{pr: v contr M wedge N = 0}\label{uso pr: v contr M wedge N = 0},
	$v\in\osp{M\wedge N}^\perp \Leftrightarrow v\lcontr (M \wedge N) = 0 \Leftrightarrow v\lcontr M = v\lcontr N=0 \Leftrightarrow v\in \osp{M}^\perp \cap \osp{N}^\perp = (\osp{M}\oplus\osp{N})^\perp$. 
	%	so $\osp{M\wedge N} = \osp{M}\oplus\osp{N}$.
	%
	\emph{(\ref{it:isp wedge})} 
	By Propositions \ref{pr:isp subset osp} and \ref{pr:inter osp 0 wedge not 0},
		\OMIT{and \ref{it:osp wedge}}
	$\isp{M\wedge N} \subset \osp{M\wedge N} = \osp{M}\oplus\osp{N}$,
	so $v\in\isp{M\wedge N} \Rightarrow v=u+w$ for $u\in\osp{M}$ and $w\in\osp{N}$.
	And
	$v\wedge M\wedge N=0 \Rightarrow w\wedge M\wedge N = -u\wedge M\wedge N \Rightarrow (u\wedge\hat{M}) \wedge (w\wedge N) = -u\wedge u\wedge M\wedge N = 0$.
		\OMIT{$w\wedge M\wedge N = -u\wedge M\wedge N$}
	As $u\wedge\hat{M}\in\bigwedge\osp{M}$ and $w\wedge N \in \bigwedge\osp{N}$, we have $u\wedge M=0$ or $w\wedge N=0$.
		\OMIT{\ref{pr:inter osp 0 wedge not 0}}
	But $\hat{M} \wedge (w\wedge N) = -(u\wedge M) \wedge N$, so both are $0$.
	Thus $u\in\isp{M}$, $w\in\isp{N}$ and $v\in \isp{M}\oplus\isp{N}$.%
\OMIT{$\oplus$ as $\isp{\cdot}\subset \osp{\cdot}$ and $\osp{M}\cap\osp{N}=\{0\}$}
\end{proof}

If $M\wedge N \neq 0$ then 
$\isp{M} \subset \isp{M\wedge N} \subset \osp{M\wedge N}$.
	\OMIT{\ref{pr:osp isp wedge}\ref{it:isp wedge}, \ref{pr:isp subset osp}}
So, $[B] \subset \osp{B\wedge N}$ if $B$ is a blade and $B\wedge N \neq 0$.
If $N\neq 0$ and $\osp{M}\cap\osp{N}=\{0\}$, $\osp{M} \subset \osp{M\wedge N}$
($M\wedge N \neq 0$ is not enough: e.g., $\osp{N} \not\subset \osp{v\wedge N}$ in \Cref{ex:wedge e inter non 0}).

Recall that $U \pperp V$ means $V^\perp \cap U \neq \{0\}$.

\begin{proposition}\label{pr:osp isp lcontr}
	For $M,N \in \bigwedge X$ we have the following, with equalities if $M,N\neq 0$ and $\osp{M}\not\pperp \isp{N}$.
		\SELFL{If $M=0$ or $N=0$ we still have $\osp{M\lcontr N} = \isp{M}^\perp\cap\osp{N} = \{0\}$.	If $M=N=0$ we have $\isp{M\lcontr N} = \osp{M}^\perp\cap\isp{N} = X$.}
	\begin{enumerate}[i)]
		\item $\osp{M\lcontr N} \subset \isp{M}^\perp\cap\osp{N}$. \label{it:osp lcontr}
			\SELF{$\osp{M}\cap\isp{N}^\perp=\{0\}$. For blades, it:space contr gave $=$ if $A\not\pperp B$}
		\item $\isp{M\lcontr N} \supset \osp{M}^\perp\cap\isp{N}$. \label{it:isp lcontr}
	\end{enumerate}
\end{proposition}
\begin{proof}
	\emph{(\ref{it:osp lcontr})}  
		\OMIT{\ref{pr:osp isp star}, \ref{pr:Hodge wedge contr}, \ref{pr:osp isp wedge}, it:star inverses}
	$\osp{M\lcontr N} = \osp{\rH{(\lH{N}\wedge M)}} = \isp{\lH{N}\wedge M}^\perp \subset (\isp{\lH{N}}+\isp{M})^\perp = \isp{\lH{N}}^\perp \cap \isp{M}^\perp = \osp{N} \cap \isp{M}^\perp$, with equality if $M,N\neq 0$ and $\osp{M} \cap \osp{\lH{N}} = \osp{M}\cap\isp{N}^\perp = \{0\}$.
	\emph{(\ref{it:isp lcontr})} Similar.
\end{proof}

\begin{proposition}\label{pr:osp isp Clifford}
	Let $M,N \in \bigwedge X$.
	\begin{enumerate}[i)]
		\item $\osp{MN} \subset \osp{M}+\osp{N}$, with equality and $\oplus$ if $M,N\neq 0$ and $\osp{M}\cap\osp{N}=\{0\}$.\label{it:osp Clifford}
		
		\item $\isp{MN} \supset (\osp{N}^\perp \cap \isp{M}) + (\osp{M}^\perp \cap \isp{N})$.
			\SELF{$\oplus$ if $M\neq 0$ or $N\neq 0$}
		For $M,N\neq 0$, 
		$\isp{MN} = \osp{N}^\perp \cap \isp{M}$ if $\osp{N} \not\pperp \isp{M}$,
		or
		$\isp{MN} = \osp{M}^\perp \cap \isp{N}$ if $\osp{M} \not\pperp \isp{N}$.%
		\label{it:isp Clifford}
	\end{enumerate}
\end{proposition}
\begin{proof}
	\emph{(\ref{it:osp Clifford})} Follows as in the proof of \Cref{pr:osp isp wedge}\ref{it:osp wedge}.
	\emph{(\ref{it:isp Clifford})} $\isp{MN} = \osp{\rH{(MN)}}^\perp = \osp{\tilde{N}\rH{M}}^\perp \supset (\osp{N} + \osp{\rH{M}})^\perp = \osp{N}^\perp \cap \isp{M}$, with equality if $M,N\neq 0$ and $\osp{N} \cap \osp{\rH{M}} = \osp{N} \cap \isp{M}^\perp = \{0\}$,
	in which case $\isp{N} \cap \osp{M}^\perp = \{0\}$.
		\OMIT{$\osp{N} \not\pperp \isp{M} \Rightarrow \isp{N} \not\pperp \osp{M}$}
	The rest follows likewise, with 	
	$\lH{(MN)} = \lH{N}\tilde{M}$.
\end{proof}

If $M,N\neq 0$, $\osp{N} \not\pperp \isp{M}$
	\OMIT{$\Rightarrow \isp{N} \not\pperp \osp{M} \Rightarrow \osp{M}^\perp \cap \isp{N} = \{0\}$}
and $\osp{M} \not\pperp \isp{N}$
	\OMIT{$\Rightarrow \isp{M} \not\pperp \osp{N} \Rightarrow \osp{N}^\perp \cap \isp{M} = \{0\}$}
then $\isp{MN} = \{0\}$.
But this only happens if $\isp{M}$, $\osp{M}$, $\isp{N}$ and $\osp{N}$ have same dimension, so $M$ and $N$ are same grade blades with $\comp{MN}{0} = \inner{\tilde{M},N} \neq 0$.

%\begin{corollary}
%	$\isp{MN} = \{0\}$ if $M,N\neq 0$, $\osp{N} \not\pperp \isp{M}$ and $\osp{M} \not\pperp \isp{N}$.
%\end{corollary}

The equivalence $[A] \pperp [B] \Leftrightarrow A \lcontr B =0$, for nonzero blades, splits as:

\begin{proposition}\label{pr:pperp isp osp}
	Let $M,N \in \bigwedge X$.
	\begin{enumerate}[i)]
		\item $\isp{M} \pperp \osp{N} \Rightarrow M\lcontr N = 0$. \label{it:isp pperp osp contr 0}
		\item $M\lcontr N = 0 \Rightarrow \osp{M} \pperp \isp{N}$, for $M,N\neq 0$. \label{it:contr 0 osp pperp isp}
	\end{enumerate}
\end{proposition}
\begin{proof}
	\emph{(\ref{it:isp pperp osp contr 0})}	
	If there is $0\neq v \in \osp{N}^\perp \cap \isp{M}$ then $v \lcontr N = 0$ and $M=v\wedge M'$ for $M'\in \bigwedge X$, so $M\lcontr N = M' \lcontr (v \lcontr N) = 0$.
		\OMIT{\ref{pr:isp alternative}\ref{it:isp factor v}}
	\emph{(\ref{it:contr 0 osp pperp isp})} If $M\lcontr N = 0$ but $\osp{M} \not\pperp \isp{N}$ then $\osp{M}^\perp\cap\isp{N} = \isp{M\lcontr N} = X$,
		\OMIT{\ref{pr:osp isp lcontr}\ref{it:isp lcontr}, \ref{pr:blade osp=isp=space}}
	so $\osp{M}=\{0\}$.
	But then $M = \lambda$ is a scalar, and $M\lcontr N = \bar{\lambda} N \neq 0$.
\end{proof}

Converses do not hold, and $\osp{M} \perp \osp{N} \not\Rightarrow M\lcontr N = 0$:
	\SELF{$\{0\} \neq \isp{M} \perp \osp{N} \Rightarrow M\lcontr N =0$ \\
		$\osp{M} \perp \osp{N} \Rightarrow M\lcontr N = \overline{\comp{M}{0}}\, N$}

\begin{example}
	Let $(v_1,\ldots,v_4)$ be orthonormal.
	For $M=v_{12} + v_{34}$ and $N=v_{12} - v_{34}$, $M\lcontr N = 0$ but $\isp{M} = \{0\}\not\pperp \osp{N}$.
	For $M=v_1+v_{12}$ and $N=v_1$, $\osp{M} = [v_{12}] \pperp [v_1] = \isp{N}$ but $M\lcontr N = 1 \neq 0$.
%	For $M=v_1+v_{23}$ and $N=v_1$, $\isp{M}=\{0\} \perp \osp{N}$ but $M\lcontr N = 1 \neq 0$.
	For $M=1+v_1$ and $N=v_2$, $\osp{M} = [v_1] \perp [v_2]=\osp{N}$ but $M\lcontr N = v_2 \neq 0$.
\end{example}

\section{Disassembling multivectors}\label{sc:Special operations}

Here we discuss special ways in which a general multivector can be written as a sum (\emph{decomposition}), exterior product (\emph{factorization}) or contraction (\emph{carving}) of other elements,
and show how such operations are related.

\subsection{Balanced and minimal decompositions}

First we consider a sum/decomposition $M = \sum_i M_i$ for $M,M_i\in \bigwedge X$.
%, that are well behaved with respect to inner or outer spaces.

\begin{proposition}\label{pr:osp isp sum}
	$\isp{M} \supset \bigcap_i\isp{M_i}$ and $\osp{M} \subset \sum_i \osp{M_i}$.
\end{proposition}
\begin{proof}
	If $v\in \isp{M_i} \ \forall i$ then $v\wedge M = \sum_i v\wedge M_i = 0$, so $v\in\isp{M}$.
	And $M_i \in \bigwedge(\osp{M_i}) \ \forall i \Rightarrow M \in \bigwedge(\sum_i \osp{M_i})$.
\end{proof}

\begin{definition}
	The sum is \emph{inner balanced} if $\isp{M} = \bigcap_i \isp{M_i}$, 
	\emph{outer balanced} if $\osp{M} = \sum_i \osp{M_i}$, 
	and \emph{balanced} if both conditions hold.
\end{definition}

Inner balanced means $\isp{M} \subset \isp{M_i} \ \forall i$, or, equivalently, $v\wedge M = 0 \Leftrightarrow v\wedge M_i=0 \ \forall i$, for $v \in X$.
Outer balanced means $\osp{M} \supset \osp{M_i} \ \forall i$, or, equivalently, $v\lcontr M = 0 \Leftrightarrow v\lcontr M_i=0 \ \forall i$.

\begin{example}\label{ex:inner xor outer}
	The sum in \Cref{ex:spaces} is outer but not inner balanced, as $\osp{M}=X$ but $\isp{M} \not\subset [v_{134}]$. 
	That of \Cref{ex:spaces 2} is inner but not outer balanced:
	$(v_5+2v_6)\lcontr M =0$ but $(v_5+2v_6) \lcontr v_{145} \neq 0$. 
\end{example}

\begin{example}\label{ex:osp isp homog decomp}
	The decomposition $M = \sum_p \comp{M}{p}$ in homogeneous components is balanced: $v\wedge M = \sum_p v\wedge \comp{M}{p} = 0 \Leftrightarrow v\wedge \comp{M}{p} = 0 \ \forall p$, as the terms have distinct grades, and likewise for $\lcontr$.
		\SELF{In particular, a sum of blades of distinct grades is balanced}
\end{example}

We can find an outer balanced blade decomposition $M = \sum_i B_i$ using a basis of $\osp{M}$.
An inner balanced one, taking $M=B \wedge N$ as in \Cref{pr:isp alternative}\ref{it:isp space max factor}, and a blade decomposition $N = \sum_i A_i$, so $M = \sum_i B \wedge A_i$.
And a balanced one, in the same way, but using \Cref{pr:M=BBM} to take $N = \frac{B \lcontr M}{\|B\|^2} \in \bigwedge\osp{M}$,\label{uso pr:M=BBM}
and decomposing it in terms of a basis of $\osp{M}$.

These decompositions can help extend results from blades to general multivectors. For example, we can generalize \Cref{pr:triple subblade}\label{uso pr:triple subblade} 
as follows, using an inner balanced blade decomposition of $N$:

\begin{proposition}\label{pr:reformulated triple prods}
	Let $A$ be a blade and $L,M,N\in\bigwedge X$.
	\OMIT{\ref{pr:triple subblade},\ref{pr:osp smallest subspace with M}, \ref{pr:osp isp min blade decomp}}
	\begin{enumerate}[i)]
		\item \label{it:B sub N}
		$(A\rcontr M)\lcontr N = (P_A M)\wedge(A\lcontr N)$, if $[A]\subset \isp{N}$.
		\SELF{$N\rcontr(L\lcontr B) = (N\rcontr B)\wedge P_B L$.}
		\item \label{it:L sub N}
		$(M\rcontr L)\lcontr N = L\wedge(M\lcontr N)$, if $\osp{L}\subset \isp{N}$.
		\CITEL{\cite[p.\,594]{Dorst2007} proves \ref{it:L sub N} for 3 blades via 3 inductions on the grades. Our proof is simpler, more geometric.}
		\SELF{Mirror $N\rcontr(M\lcontr L) = (N\rcontr M) \wedge L$ (for $L\subset N$) dá it:star wedge contr se $N=\Omega$}
		%		\item $B\wedge(B\lcontr M) = M$, if $[B]\subset \isp{M}$. \label{it:BBA subblade}
	\end{enumerate}
\end{proposition}

Star duality links inner and outer balanced decompositions:

\begin{proposition}\label{pr:balanced star}
	$\sum_i M_i$ is inner balanced $\Leftrightarrow \sum_i \rH{M_i}$ is outer balanced.
%	, and vice-versa.
\end{proposition}
\begin{proof}
	Follows from \Cref{pr:osp isp star}.
\end{proof}

\begin{definition}
	A blade decomposition of $M$ is \emph{minimal} if no blade decomposition of $M$ has fewer blades.
\end{definition}

\begin{theorem}\label{pr:osp isp min blade decomp}
	Any minimal blade decomposition is balanced.
\end{theorem}
\begin{proof}
	If $M=\sum_{i=1}^m B_i$ is not outer balanced, for some $v\in \osp{M}^\perp$ we have $I =\{i : v\lcontr B_i \neq 0\} \neq \emptyset$.
	We can assume $m \in I$.
	For $i\in I$, as $v\lcontr B_i$ is a subblade of $B_i$ we have
	$B_i = u_i \wedge (v\lcontr B_i)$ for some $u_i\in [B_i]$.
		\OMIT{pr:characterization 1}
	As
	$v \lcontr M = 0$ gives $v\lcontr B_m = -\sum_{i\in I \backslash m} v\lcontr B_i$, 
	we find
	$M = \sum_{i\notin I} B_i + \sum_{i\in I\backslash m} u_i \wedge (v\lcontr B_i) + u_m \wedge (v\lcontr B_m) =  \sum_{i\notin I} B_i + \sum_{i\in I\backslash m} (u_i-u_m) \wedge (v\lcontr B_i)$,
	with less than $m$ blades. 
	Thus any minimal blade decomposition is outer balanced,	
	and inner balanced since $\rH{M}=\sum_{i=1}^m \rH{B_i}$ is also minimal.\OMIT{\ref{pr:balanced star}}
\end{proof}

\begin{example}
	The converse is not valid: for a basis $(v_1,v_2,v_3)$, any decomposition of $M=v_1+v_{23}$ is balanced, as $\isp{M}=\{0\}$ and $\osp{M}=X$.
\end{example}

\subsection{Blade factorizations}\label{sc:Blade Factorizations}

Now we analyze which blades can be factored out of $M\in \bigwedge X$.
%, discuss special factorizations, and relate them to balanced decompositions.
\CITE{Versões mais detalhadas do 2.2.8 do Rosen e 3.8 do Kozlov2000I}

\begin{definition}\label{df:M subset N}
	An \emph{inner blade} of $M$ is a blade $B\neq 0$ 
	\SELF{Pra não ter que falar sempre que fatorar. E $B=0$ é blade de max grade (tem qualquer grade) fatorando $M=0$, mas $[0] \neq \isp{0}$}
	with $[B] \subset \isp{M}$, being \emph{maximal} if $[B]=\isp{M}$.
\end{definition}

Note that any nonzero blade is an inner blade of $M=0$.

\begin{theorem}\label{pr:B in isp factor}
	A blade $B\neq 0$ is an inner blade of $M$ if, and only if, $M=B\wedge N$ for some $N\in\bigwedge X$.
	In this case, for each complement $V$ of $[B]$ there is a unique such $N\in\bigwedge V$.
\end{theorem}
\begin{proof}
	\Cref{pr:isp alternative}\ref{it:isp space max factor} gives a blade $\mathbf{B}$ with $[\mathbf{B}] = \isp{M}$ and $M=\mathbf{B} \wedge \mathbf{N}$ for $\mathbf{N}\in \bigwedge X$.
	Any inner blade $B$ of $M$ is a subblade of $\mathbf{B}$, so $\mathbf{B} = B\wedge A$ for a blade $A$, and $M=B\wedge N$ for $N = A\wedge \mathbf{N}$.
	The converse is immediate.
	If $M=B\wedge N$, we can assume $N\in\bigwedge V$, since decomposing $N$ in terms of bases of $[B]$ and $V$, $B\wedge$ eliminates any term with elements of $[B]$.
	And if $M=B\wedge N'$ for $N'\in\bigwedge V$ then $B\wedge (N-N')=0$, so $N=N'$.
		\OMIT{\ref{pr:inter osp 0 wedge not 0}, $N-N' \in \bigwedge V$}
\end{proof}

\begin{definition}
	A \emph{blade factorization} $M=B\wedge N$, where $B\neq 0$ is a blade and $N\in\bigwedge X$, is:
	\begin{enumerate}[i)]
		\item \emph{tight} if $\osp{N} \cap [B] = \{0\}$; \label{it:def tight factor}
		\item \emph{orthogonal} if $\osp{N} \subset [B]^\perp$;
			\SELF{implies tight}
		\item \emph{maximal} if $[B]=\isp{M}$.
%		\item \emph{optimal} if it is tight and maximal.
	\end{enumerate}
\end{definition}

Tight means $N$ has no needless blades that vanish as it is appended to $B$ (via $\wedge$) to form $M$.

\begin{proposition}\label{pr:blade factor}
	Let $0 \neq M=B\wedge N$ be a blade factorization.
	\begin{enumerate}[i)]
		\item It is tight $\Leftrightarrow \osp{M} = [B]\oplus\osp{N}$ and 
		$\isp{M} = [B]\oplus\isp{N}$.\label{it:disjoint}
			\SELF{$\osp{M} = [B]\oplus\osp{N} \Rightarrow$ effic. trivially por causa de $\oplus$}
		\item It is orthogonal $\Leftrightarrow N = \frac{B\lcontr M}{\|B\|^2}$. \label{it:orthogonal}
		\item If it is maximal then $\isp{N}=\{0\}$. The converse holds if it is tight.\label{it:maximal}
%		\item If it is tight and $\isp{N}=\{0\}$ then it is maximal. \label{it:complete}
	\end{enumerate}
\end{proposition}
\begin{proof}
	\emph{(\ref{it:disjoint})} Follows from \Cref{pr:osp isp wedge}.
	\emph{(\ref{it:orthogonal})} Follows from Propositions \ref{pr:M=BBM} and \ref{pr:B in isp factor}.
	\emph{(\ref{it:maximal})} By Propositions \ref{pr:inter isp wedge 0} and \ref{pr:osp isp wedge}, $[B]\cap\isp{N}=\{0\}$ and $\isp{N}\subset \isp{M}=[B]$. The converse follows from \ref{it:disjoint}.
		\OMIT{\ref{it:isp wedge}}
	%
%	\emph{(\ref{it:complete})} Follows from \ref{it:disjoint}. 
\end{proof}

\begin{example}
	The converse of \ref{it:maximal} needs tightness:
	for a basis $(v_1,v_2,v_3)$, $v_{123} = v_1 \wedge N$ with $N=v_{1}+v_{23}$ is non-maximal but $\isp{N}=\{0\}$. 
\end{example}

\begin{proposition}\label{pr:maximal factor}
	Let $0 \neq M = B\wedge N = \mathbf{B}\wedge \mathbf{N}$ be blade factorizations, with $\mathbf{B}\wedge \mathbf{N}$ being maximal. Then $[B] \subset [\mathbf{B}]$, and if $B\wedge N$ is:
	\begin{enumerate}[i)]
		\item tight then $[\mathbf{B}]=[B]\oplus\isp{N}$; \label{it:I maximal II complem}
		\item orthogonal then $\isp{N} = [B\lcontr \mathbf{B}]$; \label{it:I maximal II orthog}
			\SELF{ $P_{[B]^\perp}(\isp{N}) = [\mathbf{B} \lcontr B]$ for non-orthogonal}
		\item maximal then $B=\lambda \mathbf{B}$ for a scalar $\lambda\neq 0$. \label{it:I maximal II maximal}
	\end{enumerate}
\end{proposition}
\begin{proof}
	$[B]\subset \isp{M} = [\mathbf{B}]$.
	\emph{(\ref{it:I maximal II complem})} Follows from \Cref{pr:blade factor}\ref{it:disjoint}.
	\emph{(\ref{it:I maximal II orthog})} $\isp{N}\subset\osp{N}\subset [B]^\perp$,
	\OMIT{\ref{pr:isp subset osp}}
	so \ref{it:I maximal II complem} gives $\isp{N} = [B]^\perp \cap [\mathbf{B}] = [B\lcontr \mathbf{B}]$.
		\OMIT{pr:characterization 1, $B \not\pperp \mathbf{B}$ as $B \subset \mathbf{B}$}
	\emph{(\ref{it:I maximal II maximal})} $[B] = [\mathbf{B}]$.
		\OMIT{$=\isp{M}$}
\end{proof}

\begin{theorem}\label{pr:unique complete factor each V}
	Given a complement $V$ of $\isp{M}$, there is a unique (up to scalars) maximal tight blade factorization $M=B\wedge N$ with $N\in\bigwedge V$.
		\SELF{so, a unique (up to scalars) maximal orthogonal blade factorization}
\end{theorem}
\begin{proof}
%	Given a maximal inner blade $B$, for $M\neq 0$ 
	Follows from Propositions \ref{pr:B in isp factor} and \ref{pr:maximal factor}\ref{it:I maximal II maximal} ($M=0$ is trivial).
\end{proof}

\begin{proposition}\label{pr:factorization decomposition}
	If a blade factorization $0\neq M=B\wedge N$ is:
	\begin{enumerate}[i)]
		\item maximal then $M=\sum_i B\wedge N_i$ is inner balanced for any decomposition $N=\sum_i N_i$; \label{it:factor decomp maximal}
		\SELF{Valeria para $M=0$}
		\item  tight then $M=\sum_i B\wedge N_i$ is outer balanced for any outer balanced decomposition $N=\sum_i N_i$. \label{it:factor decomp outer}
		\SELF{$N_i \in \bigwedge\osp{N} \ \forall i$}
		\SELF{Para valer com $M=0$ precisaria tomar $N_i=0$}
	\end{enumerate}
\end{proposition} 
\begin{proof}
	\emph{(\ref{it:factor decomp maximal})} $\isp{M} = [B] \subset \isp{B\wedge N_i}$. 
	\emph{(\ref{it:factor decomp outer})} $\osp{M} = [B] \oplus \osp{N} \supset [B] \oplus \osp{N_i} \supset \osp{B\wedge N_i}$.
		\OMIT{\ref{pr:blade factor}\ref{it:disjoint}, \ref{pr:osp isp wedge}\ref{it:osp wedge}}
\end{proof}

\begin{example}\label{ex:factorizations}
	In \Cref{ex:spaces}, $B=(v_1-v_3)\wedge(v_3+v_5)$
	\OMIT{$B = v_{13} + v_{15} - v_{35}$, $M = v_{134} - v_{145} + v_{345} + v_{1235}$}
	is a maximal inner blade, 
	and gives a maximal orthogonal factorization with $N = \frac{B\lcontr M}{\|B\|^2} = \frac{3 v_4  - v_{25} + v_{23} - v_{12}}{3}$.%
		\OMIT{$[B]\perp \osp{N}$ since $(v_1-v_3) \lcontr N = (v_3+v_5) \lcontr N = 0$}
	With $N' = v_4 + v_{23}$, a maximal tight one, as $\osp{N'} \cap [B] = [v_{234}] \cap [B] =\{0\}$.
		\OMIT{as $B \vee v_{234} = -1$ \\
		$(v_1-v_3)\lcontr N' \neq 0$ \\
		$[v_1\wedge\cdots\wedge v_p] \perp \osp{M} \Leftrightarrow v_i\lcontr M = 0 \ \forall i$, pr:regr induced basis, it:space regr}
	Since $v_4 + v_{23}$ is outer balanced, $M = B \wedge v_4 +  B \wedge v_{23}$ is balanced (indeed, a minimal blade decomposition).
		\OMIT{$M$ is not simple, as has different grades}	
	$M = B \wedge (N' + v_{15})$ is maximal but not tight, as $\osp{N' + v_{15}} = X$, and $B \wedge v_{15} = 0$.
	Also, $M = B' \wedge N''$ with $B' = v_1-v_3$ and $N'' = \frac{B' \lcontr M}{\|B\|^2} = \frac{v_{34}- 2v_{45}+v_{14}+v_{235}-v_{125}}{2}$ is orthogonal but non-maximal,
	and $\isp{N''} = [v_1+v_3+2v_5] = [B' \lcontr B]$.
\end{example}

%\begin{proposition}\label{pr:decomposition factorization}
%	If $M=\sum_i M_i$ is inner balanced and $B$ is a maximal inner blade of $M$ then $M_i = B\wedge N_i$ for $N_i \in \bigwedge ([B]^\perp)$, and the factorization $M=B\wedge \sum_i N_i$ is maximal and orthogonal.
%		\SELF{For $M=0$: $M_i = 0$ or $n$-blade, $B$ any nonzero blade}
%\end{proposition}
%\begin{proof}
%	$[B] = \isp{M} \subset \isp{M_i}$, so \Cref{pr:B in isp factor} gives the $N_i$'s.
%\end{proof}
%
%
%
%\begin{example}
%	In \Cref{ex:spaces 2,ex:inner xor outer}, 
%	$M = v_1 \wedge (v_{23} + 2 v_{45} - v_{46})$ is a maximal orthogonal factorization.
%\end{example}

\subsection{Blade carvings}\label{sc:Blade Carvings}

Next we show how to obtain $M \in \bigwedge X$ by contracting some $N$ on a blade $B$.
We call this blade carving, as $M$ is formed removing subblades of $B$ where those of $N$ project.
It dualizes the problem of blade factorization, and most results are duals of those above, being left as exercises.
	\OMIT{Ver OMMITED PROOFS; \\ Prova trocando $B$, $M$ por $\rH{B}$, $\rH{M}$ nos enunciados; \\ $N$ has no $*$}

\begin{definition}\label{df:outer blade}
	An \emph{outer blade} of $M$ is a blade $B \neq 0$ 
		\SELFL{Pra \ref{pr:outer blade carved} e \ref{pr:unique complete carv each V}, senão $B=0$ é outer blade $\forall M$ (minimal só de $M=0$), e min outer blade de $M=0$ seria $B=0$ e não $B=1$.}
	with $\osp{M} \subset [B]$, being \emph{minimal} if $[B]=\osp{M}$.
\end{definition}

\begin{definition}
	A \emph{blade carving} $M = N \lcontr B$, 
	\SELF{$\rHB{N}$}
	where $B\neq 0$ is a blade and $N\in\bigwedge X$, is:
	\begin{enumerate}[i)]
		\item \emph{tight} if $\osp{N} \cap [B]^\perp = \{0\}$; \label{it:def tight carv}
			\SELF{$N$ not pperp $B$}
		\item \emph{internal} if $\osp{N} \subset [B]$;
			\SELF{implies effic; $\osp{N} \perp [B]^\perp$}
		\item \emph{minimal} if $[B]=\osp{M}$. 
%		\item \emph{optimal} if it is tight and minimal.
	\end{enumerate}
\end{definition}

%We also say $M$ is \emph{carved} from $B$ by $N$, as it .
Tight means $N$ has no needless blades whose projections on $B$ vanish, and so do not leave any carved piece of $B$ to form $M$.

\begin{proposition}\label{pr:outer dual inner}
	Let $M,N\in\bigwedge X$ and $B\neq 0$ be a blade.
	\begin{enumerate}[i)]
		\item $B$ is an inner blade of $M \Leftrightarrow \rH{B}$ is an outer blade of $\rH{M}$.
		Moreover, $B$ is maximal $\Leftrightarrow \rH{B}$ is minimal. \label{it:outer inner}
		\item $M = B \wedge N$ is a tight (\resp orthogonal, maximal) factorization $\Leftrightarrow \rH{M} = N \lcontr \rH{B}$ is a tight (\resp internal, minimal) carving. \label{it:factorization carving}
	\end{enumerate}
\end{proposition}
\begin{proof}
	Follows from Propositions \ref{pr:Hodge wedge contr} and \ref{pr:osp isp star}.
\end{proof}

\begin{theorem}\label{pr:outer blade carved}
	A blade $B\neq 0$ is an outer blade of $M$ if, and only if, $M = N \lcontr B$ for some $N\in\bigwedge X$.
	In this case, for each complement $V$ of $[B]^\perp$
	\SELF{$V$ is subspace of maximal dim with $V \not\pperp [B]$}
	there is a unique such $N\in\bigwedge V$.
\end{theorem}

\begin{proposition}\label{pr:blade carving}
	Let $0 \neq M = N \lcontr B$ be a blade carving.
	\begin{enumerate}[i)]
		\item If it is tight then $\osp{M} = \isp{N}^\perp \cap [B]$ and $\isp{M} = \osp{N}^\perp \cap [B]$.\label{it:tight}
			\SELF{Nenhuma recíproca vale}
		\item It is internal $\Leftrightarrow N = \frac{B\rcontr M}{\|B\|^2}$. \label{it:intrinsic}
			\SELF{$=\lHB{M}$ if $B$ unit}
		\item If it is minimal then $\isp{N}=\{0\}$. The converse holds if it is tight.\label{it:minimal}\SELF{not $\osp{N}$, as \ref{pr:outer dual inner} has no $\rH{N}$; \\ if $B$ not minimal its part outside $M$ is contracted by an inner blade of $N$}
%		\item It is optimal $\Leftrightarrow$ it is tight and $\isp{N}=\{0\}$. \label{it:optimal}
	\end{enumerate}
\end{proposition}

\begin{example}
	The converse of \ref{it:minimal} needs tightness:
	for $(v_1,v_2,v_3)$ orthonormal,
	$1 = N\lcontr v_{23}$ with $N=v_1+v_{23}$ is non-minimal but $\isp{N}=\{0\}$.
%	 but $\osp{N\lcontr v_{23}} = [1] \neq[v_{23}]$. 
\end{example}

\begin{proposition}\label{pr:minimal carv}
	Let $0 \neq M = N\lcontr B = \mathbf{N} \lcontr \mathbf{B}$ be blade carvings, with $\mathbf{N} \lcontr \mathbf{B}$ being minimal. Then $[\mathbf{B}] \subset [B]$, and if $N\lcontr B$ is:
	\begin{enumerate}[i)]
		\item tight then $[\mathbf{B}] = \isp{N}^\perp \cap [B]$; \label{it:tight compare minimal}
		\item internal then $\isp{N} = [\mathbf{B} \lcontr B]$; \label{it:intrinsic compare minimal}
			\SELF{$P_B\isp{N} = [\mathbf{B} \lcontr B]$ if not internal}
		\item minimal then $B=\lambda \mathbf{B}$ for a scalar $\lambda\neq 0$. \label{it:minimal compare minimal}
	\end{enumerate}
\end{proposition}

\begin{theorem}\label{pr:unique complete carv each V}
	Given a complement $V$ of $\osp{M}^\perp$, there is a unique (up to scalars) minimal tight blade carving $M = N \lcontr B$ with $N\in\bigwedge V$.
		\SELF{so, a unique (up to scalar multiplication) minimal internal blade carving}
\end{theorem}

\begin{proposition}
	If a blade carving $0 \neq M = N \lcontr B$ is:
	\begin{enumerate}[i)]
		\item minimal then $M=\sum_i N_i \lcontr B$ is outer balanced for any decomposition $N=\sum_i N_i$; \label{it:carv decomp minimal}
		\item tight then $M=\sum_i N_i \lcontr B$  is inner balanced for any outer balanced decomposition $N=\sum_i N_i$. \label{it:carv decomp tight}
		\SELF{$N_i \in \bigwedge\osp{N} \ \forall i$}
	\end{enumerate}
\end{proposition}

\begin{example}\label{ex:carvings}
	In \Cref{ex:spaces 2}, $B = v_{1234}\wedge (v_6-2v_5)$ is a minimal outer blade, 
	and gives a minimal internal carving with 
	$N = \frac{B\rcontr M}{\|B\|^2} = \frac{v_{46} - 2v_{45} - 5v_{23}}{5}$. 
		\OMIT{$\osp{N} = [v_{234}\wedge(v_6-2v_5)] \subset [B]$ \\
			$M = v_{123} + 2 v_{145} - v_{146}$, $\osp{M} = [v_5 + 2v_6]^\perp = [v_{12346}-2v_{12345}] = [v_{1234}\wedge (v_6-2v_5)]$}
	With $N' = v_{46} - v_{23}$, a minimal tight one, 
		\OMIT{non-internal, $v_6 \in \osp{N'}$ but $v_6 \notin [B]$}
	as $\osp{N'} \cap [B]^\perp = [v_{2346}] \cap [v_5 + 2v_6] = \{0\}$.
	Since $v_{46} - v_{23}$ is outer balanced, 
	$M = v_{123} - v_{14}\wedge(v_6-2v_5)$ is balanced 
	(indeed, a minimal blade decomposition).
	$M = (N' + v_{15} + 2 v_{16}) \lcontr B$ is minimal but not tight, as $\osp{N' + v_{15} + 2 v_{16}} = X$, and $(v_{15}+2v_{16})\lcontr B = 0$.
	Also, $M = N'' \lcontr B'$ with  $B'=v_{123456}$ and $N'' = B'\rcontr M = -v_{456}-2v_{236}-v_{235}$ is internal but non-minimal,
	and $\isp{N''} = [v_5+2v_6] = [B \lcontr B']$.
\end{example}

%\begin{proposition}
%	If $M=\sum_i M_i$ is outer balanced and $B$ is a minimal outer blade of $M$ then $M_i = N_i \lcontr B$ for $N_i \in \bigwedge [B]$, and the carving $M = (\sum_i N_i) \lcontr B$ is minimal and internal. 
%		\SELF{For $M=0$: $B$ nonzero scalar, $M_i = 0$ or scalar}
%\end{proposition}
%
%
%\begin{example}
%	In \Cref{ex:spaces,ex:inner xor outer}, 
%	$M = N \lcontr v_{12345}$ with $N = v_{12345} \rcontr M = v_{25} - v_{23} + v_{12} - v_4$ is a minimal internal carving.
%\end{example}

\section{Simplicity and grades}\label{sc:Simplicity and Multivector Grades}

There is a number of known criteria that determine if a multivector can be simplified into a blade.
Our results let us easily prove them and obtain new ones.
But first we extend the concept of grade.

\subsection{Generalized grades}

For a general $0 \neq M\in\bigwedge X$, the concept of grade splits into four:

\begin{definition}\label{df:bt grades}
	The \emph{inner}, \emph{outer}, \emph{bottom} and \emph{top grades} of $M$ are
	$\igrade{M} = \dim\isp{M}$, $\ograde{M} = \dim\osp{M}$, $\bgrade{M}=\min\{p:\comp{M}{p} \neq 0\}$ and $\tgrade{M}=\max\{p:\comp{M}{p} \neq 0\}$.
	\SELF{in e out bem definidas pra $0$, excluí por simplicidade}
\end{definition}

So, $\igrade{M} = \dim \{v\in X:v\wedge M=0\}$, 
and $\ograde{M} = \codim \{v\in X:v\lcontr M=0\}$
(the \emph{rank} of $M$ \cite{Shaw1983}).
\CITE{cite[p.25]{Sternberg1964}}
$M$ is homogeneous $\Leftrightarrow \bgrade{M} = \tgrade{M}$,
and, by Propositions \ref{pr:blade osp=isp=space} and \ref{pr:isp=osp blade}, $M$ is a blade $\Leftrightarrow \igrade{M} = \ograde{M}$, 
	\OMIT{\ref{pr:isp subset osp}}
in which case all grades coincide with the usual one.

\begin{proposition}\label{pr:generalized grades}
	Let $0 \neq M\in\bigwedge X$ and $n=\dim X$.
	\begin{enumerate}[i)]
		\item $\igrade{M} \leq \bgrade{M} \leq \tgrade{M} \leq \ograde{M}$.\label{it:order grades nonzero}
		\SELF{$\igrade{0} = \dim X > 0 = \ograde{0}$, se definisse}
		\item $\igrade{M} + \ograde{\rH{M}} = \bgrade{M} + \tgrade{\rH{M}} = n$.\label{it:in + out star}
		\SELF{$= \igrade{M} + \ograde{\rH{M}}$}
	\end{enumerate}
\end{proposition}
\begin{proof}
	\emph{(\ref{it:order grades nonzero})} A maximal blade factorization $M = B\wedge N$
	\OMIT{exists by \ref{pr:unique complete factor each V}}
	gives $\igrade{M} = \grade{B} \leq \bgrade{M}$. 
	And $\tgrade{M} \leq \ograde{M}$ as $M\in \bigwedge\osp{M}$.
	\emph{(\ref{it:in + out star})} Follows from \Cref{pr:osp isp star}, and $\star:\bigwedge^p X \rightarrow \bigwedge^{n - p} X$.\OMIT{pr:Hodge geometric charac}
\end{proof}

Now we show $M+N$ is inner or outer balanced when $M$ and $N$ are distinct enough, as measured by how big $\isp{M}+\isp{N}$ is compared to their smallest top grade, or how small $\osp{M}\cap\osp{N}$ is compared to their largest bottom grade.
Later we show this implies $M+N$ is not simple.

\begin{theorem}\label{pr:isp M+N = inter isp}
	Let $b = \max\{\bgrade{M},\bgrade{N}\}$ and $t = \min\{\tgrade{M},\tgrade{N}\}$ for nonzero $M,N\in \bigwedge X$.
	\SELF{condições dão $M+N \neq 0$}
	\begin{enumerate}[i)]
		\item If $\dim(\isp{M}+\isp{N}) \geq t+2$ then $M+N$ is inner balanced
		and has $\igrade{M+N} \leq b-2$. \label{it:large isp + isp}
		\SELF{hypothesis gives $b\geq 2$}
		\item If $\dim(\osp{M}\cap\osp{N}) \leq b-2$ then $M+N$ is outer balanced
		and has $\ograde{M+N} \geq t+2$. \label{it:small osp inter osp}
	\end{enumerate}
\end{theorem}
\begin{proof}
	\emph{(\ref{it:large isp + isp})} 
	Let $U,V,W,Y\subset X$ be such that $U=\isp{M}\cap\isp{N}$, $U\oplus V=\isp{M}$, $U\oplus W=\isp{N}$ and $U\oplus V\oplus W\oplus Y=X$,
	and let	$M=A\wedge M'$ and $N=B\wedge N'$ be maximal blade factorizations with $\osp{M'}\subset W\oplus Y$ and $\osp{N'}\subset V\oplus Y$.
	
	Given $x\in\isp{M+N}$, decompose it as $x=u+v+w+y$ for $u\in U$, $v\in V$, $w\in W$ and $y\in Y$.
	As $u,v\in \isp{M}$ and $u,w\in\isp{N}$, we have
	\begin{equation}\label{eq:LI sum}
		0 = x\wedge (M+N) = (w+y)\wedge A\wedge M' + (v+y)\wedge B\wedge N'.
	\end{equation}
	In the basis of $\bigwedge X$ induced by bases
		\OMIT{possibly empty}	
	$\beta_U, \beta_V, \beta_W, \beta_Y$ of the subspaces,
%	 $\beta_U\cup\beta_V\cup\beta_W\cup\beta_Y$, 
	nonzero components of $(w+y)\wedge A\wedge M'$ have all $r = \dim V$ elements of $\beta_V$,
		\OMIT{from $A$}
	and at most $1+\tgrade{M'}$ of $\beta_W$.
	Those of $(v+y)\wedge B\wedge N'$ have at most $1+\tgrade{N'}$ of $\beta_V$, and all $s = \dim W$ of $\beta_W$.
		\OMIT{from $B$}
	%
%	But, for $q=\dim U$, we have, 
%	$\tgrade{M} = \tgrade{M'}+q+r$,
%	\OMIT{$= \tgrade{M'}+\grade{A}$}
%	$\tgrade{N} = \tgrade{N'}+q+s$ 
%	\OMIT{$= \tgrade{N'}+\grade{B}$}
	For $q=\dim U$, by hypothesis $q+r+s \geq 2 + \min\{\tgrade{M'}+q+r,\tgrade{N'}+q+s\}$, so
	$r \geq 2+\tgrade{N'}$ or $s \geq 2+\tgrade{M'}$.
%	\begin{equation}\label{eq:r s}
%		r \geq 2+\tgrade{N'} \quad \text{or} \quad s \geq 2+\tgrade{M'}.
%	\end{equation}	
	Thus there are no common components
		\OMIT{as $r>1+\tgrade{N'}$ or $s>1+\tgrade{M'}$}
	in \eqref{eq:LI sum}, and $(w+y)\wedge A\wedge M' = (v+y)\wedge B\wedge N' = 0$.
	As $[A] = U\oplus V$ and $\osp{(w+y)\wedge M'} \subset W\oplus Y$, $(w+y)\wedge M' = 0$.
		\OMIT{\ref{pr:inter osp 0 wedge not 0}}
	By \Cref{pr:blade factor}\ref{it:maximal}, $\isp{M'}=\{0\}$, so $w=y=0$.
		\OMIT{$w+y=0$, $W\cap Y=\{0\}$}
	Likewise, $v=0$, so $x=u\in \isp{M}\cap\isp{N}$.
	
	Thus $\isp{M+N} \subset \isp{M}\cap\isp{N}$, and $M+N$ is inner balanced.
		\OMIT{ref{pr:osp isp sum} gives equality}
	Then $\igrade{M+N} = \dim(\isp{M}\cap\isp{N}) 
	= \igrade{M} + \igrade{N} - \dim(\isp{M}+\isp{N}) 
	\leq \igrade{M} + \igrade{N} - t - 2 
	\leq \bgrade{M} + \bgrade{N} - \min\{\bgrade{M},\bgrade{N}\} - 2 
	= b-2$.\!\OMIT{\ref{pr:generalized grades}\ref{it:order grades nonzero}; \\ $t=\min\{\tgrade{M},\tgrade{N}\} \geq \min\{\bgrade{M},\bgrade{N}\}$}
	
	\emph{(\ref{it:small osp inter osp})} Follows from \ref{it:large isp + isp} applied to $\rH{M}$ and $\rH{N}$.
	\OMIT{\ref{pr:generalized grades}\ref{it:in + out star}}
\end{proof}

\subsection{Simplicity criteria}

Now we give conditions for a multivector to be simple, \ie a blade.

\begin{proposition}\label{pr:isp=osp iff simple}
	$M$ is simple $\Leftrightarrow \isp{M}=\osp{M}$, for $0 \neq M\in\bigwedge X$. 
		\SELF{$\Leftrightarrow v\wedge M = 0 \ \forall v\in\osp{M}$}
\end{proposition}
\begin{proof}
	Follows from Propositions \ref{pr:blade osp=isp=space} and \ref{pr:isp=osp blade}.
%	\emph{($\Rightarrow$)} Immediate.
%		\OMIT{\ref{pr:blade osp=isp=space}}
%	\emph{($\Leftarrow$)} An tight maximal blade factorization $M=B\wedge N$ 
%	\OMIT{\ref{pr:unique complete factor each V}}
%	gives $[B] = \isp{M} = \osp{M} = [B]\oplus\osp{N}$,
%		\OMIT{\ref{pr:blade factor}\ref{it:disjoint}}
%	so $N$ is a scalar. 
\end{proof}

\begin{corollary}\label{pr:Cartan criterion}
	$H\in\bigwedge^p X$ is simple $\Leftrightarrow$ the equivalent criteria below hold:%
		\OMIT{\ref{pr:isp=osp iff simple};
		for $F\in\bigwedge^{p-1} X$ and $G\in\bigwedge^{p+1} X$, $\inner{F,(H\lcontr G)\lcontr H} = \inner{(H\lcontr G)\wedge F,H} = \inner{H\lcontr G,H\rcontr F} = \inner{G,H\wedge (H\rcontr F)} = \pm\inner{G,(F\lcontr H)\wedge H}$}
	\begin{enumerate}[i)]
		\item $(F\lcontr H)\wedge H=0 \ \  \forall F\in\bigwedge^{p-1} X$.\label{it:FHH}
			\CITE{Kozlov2000I p.2246, 5.2, chama de Cartan criterion. Em p.2248 (6.2 Thm 3) usa para mostrar que blades formam subvariedade algébrica de $\bigwedge X$}
		\item $(H\lcontr G)\lcontr H = 0 \ \  \forall G\in\bigwedge^{p+1} X$. \label{it:HGH}
		\item $\inner{F\lcontr H , H\lcontr G} = 0 \ \ \forall F\in\bigwedge^{p-1} X,G\in\bigwedge^{p+1} X$.\label{it:FHHG}
	\end{enumerate}
\end{corollary}
\begin{proof}
	\Cref{pr:isp osp homog} shows \ref{it:FHH}, \ref{it:HGH} and \ref{it:FHHG} mean $\osp{H} \subset \isp{H}$, $\isp{H}^\perp \subset \osp{H}^\perp$ and $\osp{H} \perp \isp{H}^\perp$.
	By \Cref{pr:isp subset osp}, they imply $\isp{H}=\osp{H}$.
\end{proof}

With the regressive product $\vee$, we also have $(F \wedge H) \vee H =0 \ \forall F \in \bigwedge^{n-p-1} X$, and $(G \vee H) \wedge H =0 \ \forall G \in \bigwedge^{n-p+1} X$.%
	\OMIT{$H$ is simple $\Leftrightarrow \rH{H}$ is simple \\ $0 = (F\lcontr \rH{H})\wedge \rH{H} = (\rH{H} \vee \rH{F}) \wedge \rH{H} = \rH{(H \wedge F)} \wedge \rH{H} = \rH{((H \wedge F) \vee H)}$,  and \\ $0 = (\rH{H}\lcontr G)\lcontr \rH{H} = (G \vee H) \lcontr \rH{H} = \rH{(H \wedge (G \vee H))}$ \\ pr:Hodge wedge contr, pr:regressive}

\begin{corollary}\label{pr:simple bivector}
	$H\in\bigwedge^2 X$ is simple $\Leftrightarrow H\wedge H=0$.
\end{corollary}
\begin{proof}
	\emph{($\Rightarrow$)} Immediate.
	\emph{($\Leftarrow$)} $(v\lcontr H)\wedge H = \frac{v}{2}\lcontr(H\wedge H) = 0 \ \forall v\in X$.\OMIT{it:Leibniz vector grade inv, \ref{pr:Cartan criterion}}
\end{proof}

A related result \cite[p.\,174]{Pavan2017} is that a minimal blade decomposition of $0\neq H\in\bigwedge^2 X$ has $k$ blades if, and only if, $k$ is the largest integer such that $H\wedge\cdots\wedge H \neq 0$ (with $k$ $H$'s).
So, no minimal blade decomposition of $H$ has more than $\frac{\dim X}{2}$ blades.

\begin{proposition}\label{pr:4 blades}
	If  $i=\igrade{M}$ and $o=\ograde{M}$ for $0 \neq M\in\bigwedge X$,
	the homogeneous components $\comp{M}{i}$, $\comp{M}{i+1}$, $\comp{M}{o-1}$ and $\comp{M}{o}$ are simple.
		\SELF{possibly 0} 
\end{proposition}
\begin{proof}
	Given a maximal blade factorization $M=A\wedge N$ and a minimal blade carving $M= L \lcontr B$,
	let $\kappa = \comp{N}{0}$, $u = \comp{N}{1}$, $\lambda = \comp{L}{0}$ and $v = \comp{L}{1}$.
	Then
	$\comp{M}{i} = \kappa A$, $\comp{M}{i+1} = A \wedge u$, $\comp{M}{o-1} = v \lcontr B$ and $\comp{M}{o} = \bar{\lambda}\, B$.
\end{proof}

\begin{corollary}
	If $\comp{M}{p}$ is not simple then $\igrade{M}+2 \leq p \leq \ograde{M}-2$.
\end{corollary}

\begin{corollary}\label{pr:dim osp isp}
	For $0\neq H\in\bigwedge^p X$, there are 2 possibilities:
		\OMIT{\ref{pr:blade osp=isp=space}, \ref{pr:4 blades}}
	\begin{enumerate}[i)]
		\item $H$ is simple and $\igrade{H} = \ograde{H} = p$. \label{it:H simple}
		\item $H$ is not simple and $\igrade{H}+2 \leq p \leq \ograde{H}-2$. \label{it:H not simple}
	\end{enumerate}
\end{corollary}

\begin{corollary}\label{pr:dim isp + isp large homog}
	For nonzero $G,H\in\bigwedge^p X$, if $\dim(\osp{G}\cap\osp{H}) \leq p-2$ or $\dim(\isp{G}+\isp{H}) \geq p+2$ then $G+H$ is not simple.
		\SELF{\ref{pr:isp M+N = inter isp} gives more: \\
			$\dim(\isp{G}+\isp{H}) \geq p+2$ \\ $\Rightarrow G+H$ inner balanced, $\igrade{G+H} \leq p-2$; \\
			$\dim(\osp{G}\cap\osp{H}) \leq p-2$ \\ $\Rightarrow G+H$ outer balanced, $\ograde{G+H} \geq p+2$}
\end{corollary}
\begin{proof}
	Follows from Propositions \ref{pr:isp M+N = inter isp} and \ref{pr:dim osp isp}.
\end{proof}

\begin{proposition}
	For nonzero $p$-blades $A$ and $B$ with $[A]\neq[B]$, either: 
	\begin{enumerate}[i)]
		\item $\dim([A]\cap[B])+1 = p = \dim([A]+[B])-1$ and $A+B$ is simple but neither inner nor outer balanced.\label{it:p=q k=p-1 simple}
			\SELF{$[A]\cap[B] \subsetneq [A+B] \subsetneq [A]+[B]$}
		\item $\dim([A]\cap[B])+2 \leq p \leq \dim([A]+[B])-2$ and $A+B$ is not simple but is balanced.\label{it:p=q not simple}
	\end{enumerate}
\end{proposition}
\begin{proof}\OMIT{\ref{pr:isp M+N = inter isp}, \ref{pr:dim isp + isp large homog}}
	If $\dim([A]\cap[B]) = p-1$ then $A=u\wedge C$ and $B=v\wedge C$ for $u,v\in X$ and a $(p-1)$-blade $C$, with $[u]$, $[v]$, $[C]$ pairwise disjoint.
	If $\dim([A]\cap[B]) \leq p-2$ then $\dim([A]+[B]) \geq p+2$, and the result follows from Propositions \ref{pr:isp M+N = inter isp} and \ref{pr:dim isp + isp large homog}.
\end{proof}

The following less known criteria are from \cite{Eastwood2000}. 
Though based on the same idea, our proof is a little simpler and corrects an error in the original one (the decomposition of their $P$).
See \Cref{tab:symbols} for notation.

\begin{proposition}\label{pr:blade contr simple}
	For $p\geq 3$ and $1\leq r \leq p-2$,
	$H\in \bigwedge^p X$ is simple $\Leftrightarrow B\lcontr H$ is simple for every blade $B \in \bigwedge^r X$.
\end{proposition}
\begin{proof}
	($\Leftarrow$) Let $r=1$ and $0 \neq v_1 \in \osp{H}$. 
	By hypothesis, $0 \neq v_1 \lcontr H = v_{2\cdots p}$ with 
	%	$v_2,\ldots,v_p \in [v_1]^\perp$. Assume w.l.o.g.\ 
	$(v_1,\ldots,v_p)$ orthonormal (w.l.o.g.).
	Let $(w_1,\ldots,w_q)$ be a basis of $[v_{1\cdots p}]^\perp$.
	As $v_1\lcontr (H-v_{1\cdots p})=0$, we have $\osp{H-v_{1\cdots p}} \subset [v_1]^\perp$, so%
	\begin{equation*}
		H = v_{1\cdots p} + w\wedge v_{2\cdots p} + \sum_{\ii\in\II^q, 1 < |\ii| < p} F_\ii \wedge w_\ii + G,
	\end{equation*}
	for $w \in [w_{1\cdots q}]$, $F_\ii \in \bigwedge^{p-|\ii|} [v_{2\cdots p}]$ and $G\in \bigwedge^p [w_{1\cdots q}]$.
	For distinct $j,k \in \{2,\ldots,p\}$, and with $'$ indicating an index is absent, we have
	$v_j = \pm v_{1\cdots j'\cdots p} \lcontr H = \pm v_{1\cdots j' \cdots k' \cdots p} \lcontr (v_k \lcontr H) \in \osp{v_k \lcontr H} = [v_k \lcontr H]$,
	\OMIT{aqui usa orthonormalidade}
	by hypothesis.
	Thus $0 = v_j \wedge (v_k \lcontr H) = \sum_{\ii\in\II^q, 1 < |\ii| < p}  v_j \wedge (v_k \lcontr F_\ii) \wedge w_\ii$, 
		\OMIT{$v_j \wedge (v_k \lcontr (v_{1\cdots p} + w\wedge v_{2\cdots p})) = 0$ due to $v_j \wedge$, and $v_j \wedge (v_k \lcontr G) = 0$ due to $v_k\lcontr$}
	and so $v_j \wedge (v_k \lcontr F_\ii) = 0 \ \forall\, \ii$ and $j\neq k$.
	Then $[v_{2\cdots k' \cdots p}] \subset \isp{v_k \lcontr F_\ii}$, which requires $v_k \lcontr F_\ii=0$, as its grade is $p-|\ii|-1 < p-2$.
	Thus $\osp{F_\ii} \subset [v_{2\cdots p}] \subset \osp{F_\ii}^\perp$, and so $F_\ii = 0 \ \forall \ii$.
	We are left with $H = (v_1+w) \wedge v_{2\cdots p} + G$.
	But if $G\neq 0$, for $0\neq u \in \osp{G} \cap [w]^\perp$
		\OMIT{existe pois $\dim \osp{G}\geq p \geq 3$}
	we would have $(v_1+u)\lcontr H = v_{2\cdots p} + u\lcontr G$, which for $p\geq 3$ is not simple, contradicting the hypothesis.
	For $r>1$, the result follows by induction, using $(B\wedge v) \lcontr H = v \lcontr (B \lcontr H)$.
		\OMIT{If holds for $<r$, $v\lcontr(A\lcontr H) = (A\wedge v) \lcontr H$ simples $\forall v\in X, A\in \bigwedge^{r-1} \Rightarrow A\lcontr H$ simples $\forall A \Rightarrow H$ simples}
	($\Rightarrow$) Immediate.
\end{proof}

\begin{proposition}\label{pr:Cartan criterion 2}
	$H\in\bigwedge^p X$ is simple $\Leftrightarrow$ the equivalent criteria below hold:
	\begin{enumerate}[i)]
		\item $(F\lcontr H)\wedge H=0 \ \ \forall F\in\bigwedge^{p-2} X$.
		\item $(H\lcontr G)\lcontr H = 0 \ \  \forall G\in\bigwedge^{p+2} X$.
		\item $\inner{F\lcontr H , H\lcontr G} = 0 \ \  \forall F\in\bigwedge^{p-2} X,G\in\bigwedge^{p+2} X$.\label{it:Cartan 2-3}
	\end{enumerate}
\end{proposition}
\begin{proof}
	Equivalence is due to 
	$\inner{G,(F\lcontr H)\wedge H} 
	= \inner{G\rcontr H, F \lcontr H}
	= \inner{F \wedge (G\rcontr H), H}
	= \inner{F, H \rcontr (G\rcontr H)}$.
%	$\inner{F,(H\lcontr G)\lcontr H} = \inner{(H\lcontr G)\wedge F,H} = \inner{H\lcontr G,H\rcontr F} = \inner{G,H\wedge (H\rcontr F)}$.
	\OMIT{$= \pm\inner{G,(F\lcontr H)\wedge H}$}
	($\Rightarrow$) Immediate.
	($\Leftarrow$) 
	For any blade $B\in\bigwedge^{p-2} X$ we have $(B\lcontr H) \wedge (B\lcontr H) = \pm B\lcontr ((B\lcontr H) \wedge H) = 0$,
	so $B\lcontr H$ is simple, by \Cref{pr:simple bivector}.
	The result follows from \Cref{pr:blade contr simple}.
\end{proof}

These criteria can be easier to check than those of \Cref{pr:Cartan criterion}, since
$\dim \bigwedge^{p-2} X < \dim \bigwedge^{p-1} X$ or $\dim \bigwedge^{p+2} X < \dim \bigwedge^{p+1} X$.

\subsubsection{Plücker relations}

Let $H = \sum_{\ii\in\II^n_p} \lambda_\ii v_\ii \in\bigwedge^p X$ for an orthonormal basis $(v_1,\ldots,v_n)$ of $X$ and scalars $\lambda_\ii$.
The \emph{Plücker relations} \cite{Gallier2020,Jacobson1996} 
	\CITE{Gallier2020 p.119, quase \eqref{eq:Plucker} \\
	Jacobson1996 p.110 formato \eqref{eq:Plucker 1}}
are quadratic equations the $\lambda_\ii$'s satisfy if, and only if, $H$ is simple.
They describe an embedding of the Grassmannian of $p$-dimensional subspaces of $X$ as an algebraic variety in the projective space $\mathds{P}(\bigwedge^p X)$,
and can be written (see \Cref{tab:symbols}) as%
%\footnote{See \Cref{tab:symbols} for notation.}
\begin{equation}\label{eq:Plucker 1}
	\sum_{i=0}^{p} (-1)^i \lambda_{\jj k_i} \lambda_{\kk \backslash k_i} = 0 \ \  \forall\, \jj\in\II^n_{p-1}, \kk=(k_0,\ldots,k_{p}) \in \II^n_{p+1},
\end{equation}
if we set $\lambda_{\jj k_i} = 0$ if $k_i \in \jj$, otherwise $\lambda_{\jj k_i} = \epsilon_{\jj k_i} \lambda_{\jj \cup k_i}$ (so it switches sign with each index transposition).
In the end, one must order the indices to simplify repeated terms.
Another formulation, that avoids this, is%
\begin{equation}\label{eq:Plucker}
	\sum_{k\in\kk \backslash \jj} (-1)^{\pairs{\jj\triangle\kk}{k}} \lambda_{\jj\cup k} \lambda_{\kk \backslash k} = 0 \ \ \forall\, \jj\in\II^n_{p-1},\kk\in\II^n_{p+1}.
\end{equation} 
It is proven as \eqref{eq:Plucker 2} below, with \Cref{pr:Cartan criterion}.
	\OMIT{\ref{it:FHHG}}
Equivalence with \eqref{eq:Plucker 1} is due to $(-1)^{\pairs{\jj\triangle\kk}{k}} = (-1)^{\pairs{\jj}{k} + \pairs{\kk}{k}} = \epsilon_{\jj k} \cdot (-1)^{\pairs{\kk}{k}}$,
by \Cref{pr:epsilon}. \label{uso pr:epsilon}

There is much redundancy in this system.
If $\jj\subset\kk$ then $\kk=\jj\cup k_1 k_2$ for $k_1,k_2 \in \kk \backslash \jj$ with $k_1<k_2$, and we obtain $\lambda_{\jj\cup k_1}\lambda_{\jj\cup k_2} - \lambda_{\jj\cup k_2}\lambda_{\jj\cup k_1} = 0$,
which is trivial.
If $|\jj\cap\kk|=p-2$,
\OMIT{$\jj=(\jj\cap\kk)\cup k$ and $\kk=(\jj\cap\kk)\cup(k_1 k_2 k_3)$}
swapping the only index of $\jj \backslash (\jj\cap\kk)$ and each of the 3 indices of $\kk \backslash (\jj\cap\kk)$ we find 4 equivalent equations.
For $|\jj\cap\kk|<p-2$ all equations are nonequivalent.

A reduced set of equations that can play the same role is:

\begin{theorem}\label{pr:new Plucker}
	$H$ is simple if, and only if,
	\begin{equation}\label{eq:Plucker 2}
		\sum_{\substack{\kk_2 \in \II_2^n, \kk_2 \subset \kk \backslash \jj}} (-1)^{\pairs{\jj\triangle\kk}{\kk_2}} \lambda_{\jj\cup \kk_2} \lambda_{\kk \backslash \kk_2} = 0 \ \ \forall\, \jj\in\II^n_{p-2},\kk\in\II^n_{p+2}.
	\end{equation} 
\end{theorem}
\begin{proof}
	\Cref{pr:Cartan criterion 2}\ref{it:Cartan 2-3} means $\sum_{\ii,\ll\in\II^n_p} \bar{\lambda}_\ii \bar{\lambda}_\ll \inner{v_\jj \lcontr v_\ii,v_\ll \lcontr v_\kk} = 0 \ \forall\, \jj\in\II^n_{p-2}$, $\kk\in\II^n_{p+2}$.%
		\OMIT{No final conjuga a equação}
	The inner product is $0$ unless
	$\ii=\jj\cup \kk_2$ and $\ll=\kk \backslash \kk_2$
	for $\kk_2\in\II^n_2$ with $\kk_2 \subset \kk \backslash \jj$,
	in which case 
	$\inner{v_\jj \lcontr v_\ii,v_\ll \lcontr v_\kk} = 
%	$v_\jj \lcontr v_\ii = \epsilon_{\jj \kk_2} v_{\kk_2}$ and $v_\ll \lcontr v_\kk = \epsilon_{(\kk \backslash \kk_2)\kk_2} v_{\kk_2}$ .
%	And $
	\epsilon_{\jj \kk_2} \epsilon_{(\kk \backslash \kk_2)\kk_2} = (-1)^{\pairs{\jj}{\kk_2}+\pairs{\kk \backslash \kk_2}{\kk_2}} = (-1)^{\pairs{\jj}{\kk_2}+\pairs{\kk}{\kk_2}-1} = (-1)^{\pairs{\jj\triangle\kk}{\kk_2}-1}$.%
		\OMIT{\ref{pr:epsilon}, eq:contr induced basis}
\end{proof}

Now $\jj\subset\kk$ no longer gives a trivial equation, but each term appears twice in the sum.
If $|\jj\cap\kk|=p-3$,
\OMIT{$\jj=(\jj\cap\kk)\cup i$ and $\kk=(\jj\cap\kk)\cup(i_1 i_2 i_3 i_4 i_5)$}
swapping the only index of $\jj \backslash (\jj\cap\kk)$ and each of the 5 indices of $\kk \backslash (\jj\cap\kk)$ we find 6 equivalent equations.
For $|\jj\cap\kk|<p-3$ all equations are nonequivalent.
Though \eqref{eq:Plucker 2} can give less equations than \eqref{eq:Plucker}, some can be more complex.
Setting $\lambda_{\jj k_i k_l} = 0$ if $k_i k_l \cap \jj \neq \emptyset$, otherwise $\lambda_{\jj k_i k_l} = \epsilon_{\jj k_i k_l} \lambda_{\jj \cup k_i k_l}$,
we can rewrite it as
\begin{equation}\label{eq:Plucker 2b}
	\sum_{0\leq i < l \leq p+1} (-1)^{i+l} \lambda_{\jj k_i k_l} \lambda_{\kk \backslash k_i k_l} = 0 \ \ \forall \jj\in\II^n_{p-2}, \kk=(k_0,\ldots,k_{p+1})\in\II^n_{p+2}.
\end{equation}

\begin{example}\label{ex:Plucker 2 4}
	For $(p,n)=(2,4)$, we find
	$\lambda_{12}\lambda_{34}-\lambda_{13}\lambda_{24}+\lambda_{14}\lambda_{23}=0$
	in \eqref{eq:Plucker}
	for $(\jj,\kk) = (1,234)$, $(2,134)$, $(3,124)$ or $(4,123)$,
		\SELF{Up to sign. \\ 
			$(\jj,\kk) = (1,123)$ gives trivial $-\lambda_{12}\lambda_{13} + \lambda_{13}\lambda_{12} = 0$}
	and the other equations are trivial.
	We find the same in \eqref{eq:Plucker 1} after ordering the indices.
	In \eqref{eq:Plucker 2} or \eqref{eq:Plucker 2b}, the only possibility is $(\jj,\kk) = (\emptyset,1234)$, and gives the same equation (with each term twice repeated: e.g., $\kk_2=12$ or $34$ give $-\lambda_{12}\lambda_{34}$).
	It describes the Grassmannian as an algebraic variety of dimension $p(n-p)=4$ in a projective space of dimension $\binom{n}{p}-1=5$.
\end{example}

\begin{example}
	For $(p,n)=(2,5)$, eliminating the redundancies in \eqref{eq:Plucker} leaves a minimal set of 5 relations, for $(\jj,\kk) = (1,234)$, $(1,235)$, $(1,245)$, $(1,345)$ and $(2,345)$.
		\SELF{number minimal relations 
			$\binom{n}{p}\!\! \left[\frac12\binom{n}{p}+\frac12-\frac{1}{p+1}\binom{n+1}{p} \right]$ \\ \href{https://mathoverflow.net/questions/230710/number-of-pl\%C3\%BCcker-relations-for-a-grassmannian}{link mathoverflow number of plucker relations}}
	In \eqref{eq:Plucker 2}, $\jj=\emptyset$ and the 5 choices of $\kk\in \II_4^5$ give the same equations,
		\SELF{with terms twice repeated}
	$\lambda_{ij}\lambda_{kl} - \lambda_{ik}\lambda_{jl} + \lambda_{il}\lambda_{jk} = 0$ for $1\leq i<j<k<l \leq 5$.
		\CITE{Rosen p.69}
	Though $\dim \mathds{P}(\bigwedge^2 \R^5)=9$, the Grassmannian has dimension $6$, as it is not a complete intersection of their solution sets (in a rough analogy, it is like when 3 surfaces in $\R^3$ intersect along the same curve).
\end{example}

\begin{example}\label{ex:Plucker 3 5}
	$(p,n)=(3,5)$ dualizes the previous example.
	A minimal set in \eqref{eq:Plucker} is given by $(\jj,\kk) = (12,1345)$, $(12,2345)$, $(13,2345)$, $(14,2345)$ and $(15,2345)$.
	The same set appears in \eqref{eq:Plucker 2} for $\kk = 12345$ and $1\leq j\leq 5$.
	For $j=3$, it gives $\lambda_{123}\lambda_{345} - \lambda_{134}\lambda_{235} + \lambda_{135}\lambda_{234} = 0$,
	which occurs in \eqref{eq:Plucker} for $(13,2345)$, $(23,1345)$, $(34,1235)$ or $(35,1234)$,
	and in \eqref{eq:Plucker 2b} once we order 
	$\lambda_{312}\lambda_{345} + \lambda_{314}\lambda_{235} - \lambda_{315}\lambda_{234} - \lambda_{234}\lambda_{135} + \lambda_{325}\lambda_{134} + \lambda_{345}\lambda_{123} = 0$.
\end{example}

\begin{example}
	For $(p,n)=(3,6)$, \eqref{eq:Plucker} gives 
		\OMIT{after eliminating redundancies}
	30 equations with 3 terms, and 15 with 4.
	And \eqref{eq:Plucker 2} gives 30 with 3 terms, and 1 with 10: e.g.,
	$(\jj,\kk)=(1,12345)$ gives $\lambda_{123}\lambda_{145} - \lambda_{124}\lambda_{135} + \lambda_{125}\lambda_{134}=0$,
	while
	$(1,23456)$, $(2,13456), \ldots$ or $(6,12345)$ give $\lambda_{123}\lambda_{146} - \lambda_{124}\lambda_{356} + \lambda_{125}\lambda_{346} - \lambda_{126} \lambda_{345} + \lambda_{134}\lambda_{256} - \lambda_{135}\lambda_{246} + \lambda_{136}\lambda_{245} + \lambda_{145}\lambda_{236} - \lambda_{146}\lambda_{235} + \lambda_{156}\lambda_{234} = 0$.
\end{example}

\section{Supersymmetry}\label{sc:Supersymmetry}

Here we discuss how contractions relate to supersymmetry \cite{Oziewicz1986,Shaw1983,Varadarajan2004}.
% and to multi-fermion creation and annihilation operators \cite{,}.

\begin{definition}\label{df:superalgebra}
	A \emph{superalgebra} $\AA$ is a $\Z_2$-graded algebra, \ie it decomposes as $\AA=\AA_0\oplus\AA_1$ and its product respects  $\AA_i\AA_j\subset\AA_{i+j}$ for $i,j\in\Z_2$.
	The subspaces
		\SELF{$\AA_1$ is not subalgebra}
	$\AA_0$ and $\AA_1$, and their elements, are \emph{even} and \emph{odd}, and $T\in\AA_i$ has \emph{parity} $i$. 
		\SELF{$0$ is even and odd}
	The \emph{supercommutator} $\scom{S,T} = ST-(-1)^{ij}TS$, for $S\in\AA_i$\label{df:supercommutator}
	and $T\in\AA_j$, is extended linearly for all $S,T\in\AA$.
		\SELF{$\scom{S,T} = -(-1)^{st}\scom{T,S}$. \\ \emph{Super-anticommutator} is $\sacom{S,T} = ST+(-1)^{st}TS$. \\ $\sacom{S,T} = (-1)^{st}\sacom{T,S}$}
\end{definition}

If $S$ or $T$ is even, $\scom{S,T}$ is the usual commutator $[S,T] = ST-TS$.
If both are odd, it is the usual anticommutator $\{S,T\} = ST+TS$. 
	\SELF{$S=S_0+S_1$, $T=T_0+T_1$ for $S_0,T_0\in\AA_0$, $S_1,T_1\in\AA_1$, \\ $\scom{S,T} = [S_0,T_0] + [S_0,T_1]+[S_1,T_0]+\{S_1,T_1\}$, \\ $\sacom{S,T} = \{S_0,T_0\} + \{S_0,T_1\} + \{S_1,T_0\} + [S_1,T_1]$}

\begin{example}
	$\bigwedge X$ is a superalgebra, with even and odd subspaces $\bigwedge^+ X=\bigoplus_{k\in\N}\bigwedge^{2k} X$ and $\bigwedge^- X=\bigoplus_{k\in\N}\bigwedge^{2k+1} X$, respectively.
	It is \emph{supercommutative}, \ie $\scom{M,N}=0$ for all $M,N\in \bigwedge X$.
\end{example}

\begin{example}
	$\End(\AA)=\End_0(\AA)\oplus\End_1(\AA)$ is a superalgebra,
		\SELF{$\End(\AA)$ with $\scom{.,.}$ is a super-Lie algebra [p.\,89]{Varadarajan2004}}
	where $\End_i(\AA)$
%	= \{T\in\End(\AA):  T(\AA_j)\subset\AA_{i+j} \,\forall j\in\Z_2\}$ 
	is the set of endomorphisms of parity $i$ (that map $\AA_j\rightarrow\AA_{i+j}$ $\forall j\in\Z_2$).
	Indeed, if $P_i:\AA\rightarrow\AA_i$ is the projection, $T\in\End(\AA)$ decomposes as $T=T_0 + T_1$ with $T_0 = P_0TP_0 + P_1TP_1\in\End_0(\AA)$ and $T_1 = P_1TP_0 + P_0TP_1\in\End_1(\AA)$.
		\OMIT{$T=(P_0+P_1)T(P_0+P_1)$}
	Even endomorphisms preserve parities 
%	(mapping $\AA_i \rightarrow\AA_i$), 
	and odd ones switch them,
%	($\AA_i\rightarrow\AA_{1-i}$),
	so their compositions respect the grading.
\end{example}

%\subsection{Operators of exterior and interior products}\label{sc:Operators of Exterior and Interior Products}

So, $\End(\bigwedge X)$ is a superalgebra.
In \cite{Mandolesi_Contractions}, we studied exterior and interior product operators
$\ext_M, \iota_M\in\End(\bigwedge X)$ given by:

\begin{definition}
	$\ext_M (N) = M\wedge N$ and $\iota_M (N) = M\lcontr N$, for $M,N\in\bigwedge X$.
\end{definition}

They are adjoint operators, and have the parity (or lack thereof) of $M$.
	\SELF{For $H\in\bigwedge^p X$, $\ext_H$ and $\iota_H$ are homogeneous of degrees $p$ and $-p$, respectively, mapping $\bigwedge^k X \rightarrow \bigwedge^{k\pm p} X$. \\
	$\ext_H=0$ on $\bigwedge^k X$ for $k>n-p$, $\iota_H=0$ for $k<p$} 
As $\ext_M\ext_N = \ext_{M\wedge N}$ and $\iota_M\iota_N = \iota_{N\wedge M}$, 
the maps $e, \iota:\bigwedge X \rightarrow \End(\bigwedge X)$ given by $\ext(M) = \ext_M$ and $\iota(M) = \iota_M$ are, respectively, a superalgebra homomorphism
and a conjugate linear antihomomorphism.
The supercommutators in \ref{it:scom ee ii} below reflect this and the supercommutativity of $\bigwedge X$;
\ref{it:scom int_v} corresponds to
 the graded Leibniz rule and its adjoint;
and \ref{it:scom ie} and \ref{it:scom ei} are their higher order versions.
See \Cref{tab:symbols} for notation.
%Recall that $\ii' = (1,\ldots,p) \backslash \ii$ for $\ii\in\II^p$.

\begin{proposition}\label{pr:exterior and interior}
	Let $M,N\in\bigwedge X$ and $v,v_1,\ldots,v_p\in X$.
	\begin{enumerate}[i)]
%		\item $\ext_M\ext_N = \ext_{M\wedge N}$ and $\iota_M\iota_N = \iota_{N\wedge M}$. \label{it:ext ext int int}
%		\item $\ext_M^2 = \iota_M^2=0$ if $M$ is odd or a non-scalar blade. \label{it:ext int square} 	
%		\item $\ext_M^\dagger = \iota_M$.\label{it:adjoint ext int prods}
		\item $\scom{\ext_M,\ext_N} = \scom{\iota_M,\iota_N} = 0$. \label{it:scom ee ii}
			\SELF{$\sacom{\ext_M,\ext_N} = 2\,\ext_{M\wedge N}$, \\ $\sacom{\iota_M,\iota_N} = 2\,\iota_{N\wedge M}$}
		\item $\scom{\iota_v,\ext_M} = \ext_{v\lcontr M}$ and $\scom{\ext_v,\iota_M} = \iota_{M\rcontr v}$. \label{it:scom int_v}
		\item $\scom{\iota_{v_{1\cdots p}},\ext_M} =  \sum_{\emptyset \neq \ii\in\II^p} \epsilon_{\ii'\ii}\, \ext(\hhat{M}{p} \rcontr \hat{v}_{\ii})\,\iota_{v_{\ii'}}$. \label{it:scom ie}
			\SELFL{Alternatives: \\ $\scom{\iota_{v_{1\cdots p}},\ext_M} =  \sum\limits_{\emptyset \neq \ii\in\II^p} \epsilon_{\ii\ii'}\, \ext(\hhat{(v_{\ii}\lcontr M)}{|\ii'|}) \iota_{v_{\ii'}}  =  \sum_{\ii\in\II^p, |\ii|\neq p} \epsilon_{\ii\ii'}\, \ext(v_{\ii'}\lcontr \hhat{M}{|\ii|})\,\iota_{v_\ii}$ \\ \, \\ Adjoints: \\ $\scom{\iota_M,\ext_{v_{1\cdots p}}} =  \sum_{\emptyset \neq \ii\in\II^p} \epsilon_{\ii'\ii}\, \ext_{v_{\ii'}}\,\iota(v_{\ii}\lcontr \hhat{M}{|\ii'|})$ \\ $\scom{\ext_M,\iota_{v_{1\cdots p}}} =  \sum_{\emptyset \neq \ii\in\II^p} \epsilon_{\ii\ii'}\, \iota_{v_\ii'}\,\ext(\hat{v}_{\ii}\lcontr \hhat{M}{p})$}
		\item $\scom{\ext_{v_{1\cdots p}},\iota_M} =  \sum_{\emptyset \neq \ii\in\II^p} \epsilon_{\ii\ii'}\, \iota(\hat{v}_{\ii} \lcontr \hhat{M}{p})\,\ext_{v_{\ii'}}$. \label{it:scom ei}
%		\item $\scom{\iota_B,\ext_M} = \ext_{B\lcontr M}\, + \sum_{k=1}^{p-1} \sum_{\ii\in\II_k^p} (-1)^{\frac{k(k+1)}{2} + \|\ii\|}  \ext_{v_{\ii'}\lcontr \hhat{M}{k}}\,\iota_{v_\ii}$.\label{it:scom ie}
%		\item $\scom{\iota_M,\ext_B} =  \sum\limits_{\ii\in\II^p, |\ii|\neq p} \epsilon_{\ii\ii'} \ext_{v_\ii}\,\iota_{v_{\ii'}\lcontr \hhat{M}{|\ii|}}$. \label{it:scom iMeB}
	\end{enumerate}
\end{proposition}
\begin{proof}
	\emph{(\ref{it:scom ee ii})} Immediate. 
		\OMIT{$\ext_M\ext_N = \ext_{M\wedge N}$ and $\iota_M\iota_N = \iota_{N\wedge M}$, or from supercommut. of $\bigwedge X$ and superalgebra (anti)homomorph.}
	\emph{(\ref{it:scom int_v})} This is \Cref{pr:contr homog}.\label{uso2 pr:contr homog}
	\emph{(\ref{it:scom ie},\ref{it:scom ei})} These are \Cref{pr:higher order},\label{uso pr:higher order}
	with the term for $\ii=\emptyset$
		\OMIT{$(-1)^{pq}\ext_M\iota_{v_{1\cdots p}}$ for $M \in \bigwedge^q X$}
	moved to the left.
		\OMIT{Can assume $M\in \bigwedge^q X$. \\ $\iota_B\ext_M =  \sum_{\ii\in\II^p} \epsilon_{\ii\ii'} \ext_{v_{\ii'}\lcontr \hhat{M}{|\ii|}}\,\iota_{v_\ii}
		= (-1)^{pq}\ext_M\iota_B + \sum_{\ii\in\II^p, \ii\neq 1\cdots p} \epsilon_{\ii\ii'} \ext_{v_{\ii'}\lcontr \hhat{M}{|\ii|}}\,\iota_{v_\ii}$}
%	\emph{(\ref{it:scom ei})} This is \Cref{pr:higher order}\ref{it:v1p wedge M contr N},
%	with the term for $\ii=\emptyset$ moved to the left side.
%	
%	\emph{(\ref{it:ext ext int int})} The first formula is immediate, the other follows from \Cref{pr:contr homog}\ref{it:wedge contractor}.
%	%
%	\emph{(\ref{it:ext int square})} Immediate.
%	%
%	\emph{(\ref{it:scom ee ii})} For $M\in\bigwedge^p X$ and $N\in\bigwedge^q X$ we have $\ext_M\ext_N = \ext_{M\wedge N} = (-1)^{pq}\ext_N \ext_M$,  and likewise for $\iota_M\iota_N$.
%	%
%	\emph{(\ref{it:scom int_v})} The first formula follows from \Cref{pr:contr homog}\ref{it:Leibniz vector grade inv}, and the other is its adjoint.
%	%
%	\emph{(\ref{it:scom ie})} 
%	For $M\in\bigwedge^q X$, \Cref{pr:generalized Leibniz} gives
%	$\iota_B\ext_M 
%	=  \sum_{\ii\in\II^p} \epsilon_{\ii\ii'} \ext_{v_{\ii'}\lcontr \hhat{M}{|\ii|}}\,\iota_{v_\ii}
%	= (-1)^{pq}\ext_M\iota_B + \sum_{\ii\in\II^p, \ii\neq 1\cdots p} \epsilon_{\ii\ii'} \ext_{v_{\ii'}\lcontr \hhat{M}{|\ii|}}\,\iota_{v_\ii}$.
%%	\emph{(\ref{it:scom iMeB})} Likewise, with \Cref{pr:generalized Leibniz adjoint}.
\end{proof}

%Note the reordering of $M$ and $N$ in $\iota_{N\wedge M}$, in \ref{it:ext ext int int}.
In particular, $\scom{\iota_v,\ext_w}  = \inner{v,w}\mathds{1}$
	\SELF{$= \{\iota_v,\ext_w\}$}
for $v,w\in X$.

%The adjoints $\scom{\iota_M,\ext_{v_{1\cdots p}}}$ and $\scom{\ext_M,\iota_{v_{1\cdots p}}}$ of \ref{it:scom ie} and \ref{it:scom ei} correspond to other formulas in \cite{Mandolesi_Contractions}.

\begin{example}\label{ex:supercommutator}
	$\scom{\iota_{v_{12}},\ext_M} = \ext(M \rcontr v_{12})  + \ext(M \rcontr v_1)\iota_{v_2} - \ext(M \rcontr v_2)\iota_{v_1}$,
	and 
	$\scom{\ext_{v_{12}},\iota_M} = - \iota(v_1\lcontr M)\ext_{v_2} + \iota(v_2\lcontr M)\ext_{v_1} + \iota(v_{12}\lcontr M)$.
		\OMIT{\ref{pr:exterior and interior}\ref{it:scom ie}}
\end{example}

%\begin{corollary}\label{pr:scom int_v}
%	$\scom{\iota_v,\ext_M} = \ext_{v\lcontr M}$ and $\scom{\iota_M,\ext_v} = \iota_{v\lcontr M}$ for $v \in X$ and $M\in \bigwedge X$.
%%	, for $v\in X$ and $M\in\bigwedge X$.
%\end{corollary}
%\begin{proof}
%	Follows from Propositions \ref{pr:contr homog}\ref{it:Leibniz vector grade inv} and \ref{pr:v wedge (M contr N)}.
%\end{proof}

\subsection{Creation and annihilation operators}\label{sc:Creation and Annihilation Operators}

Fix an orthonormal basis $(v_1,\ldots,v_n)$, and let $\rr,\ss\in\MM^n$ (see \Cref{tab:symbols}).

\begin{definition}\label{df:annihilation creation}
	$a_\rr^\dagger = \ext_{v_\rr}$ and $a_\rr = \iota_{v_\rr}$
		\SELFL{precisa $\rr$ para concatenations}
	are \emph{creation} and \emph{annihilation operators}, respectively.
		\SELF{$\lH{(a_\ii v_\jj)} = (-)^{\ii(n-\jj)} a_\ii^\dagger \lH{v_\jj}$ \\ $\lH{(a_\ii^\dagger v_\jj)} \!\!=\!\! (-)^{\ii(n-\jj-\ii)} a_\ii \lH{v_\jj}$ \\ $\rH{(a_\ii v_\jj)} = (-)^{\ii(\jj+1)} a_\ii^\dagger \rH{v_\jj}$ \\ $\rH{(a_\ii^\dagger v_\jj)} = (-)^{\ii\jj} a_\ii \rH{v_\jj}$}
\end{definition}

They are adjoints, with $a_\rr^\dagger v_\ss = \delta_{\rr\cap \ss = \emptyset}\, v_{\rr\ss}$
and 
$a_\rr v_\ss = \delta_{\rr \subset \ss} \epsilon_{\rr(\ss \backslash \rr)} v_{\ss \backslash \rr}$ (so $a_\rr v_{\rr\ss} = v_\ss$ if $\rr \cap \ss = \emptyset$).
Also,
$a_\rr^\dagger a_\ss^\dagger = \delta_{\rr\cap \ss = \emptyset}\, a_{\rr\ss}^\dagger$ and $a_\rr a_\ss = \delta_{\rr\cap \ss = \emptyset}\, a_{\ss\rr}$ (with $\rr$ and $\ss$ swapped),
so
$(a_\rr^\dagger)^2 = a_\rr^2 = 0$ if $\rr\neq\emptyset$.
Note that $a_\emptyset^\dagger = a_\emptyset = \Id$.
For a single index $i$, $a_i^\dagger a_i + a_i a_i^\dagger = \Id$.

These properties explain the terminology, from Physics \cite{Oziewicz1986}:
if the $v_i$'s are single-fermion quantum states, $v_\ss$ is a multi-fermion state with a particle in each $v_i$ for $i\in\ss$;
acting on it, $a_\rr^\dagger$ and $a_\rr$ add or remove fermions in the states specified by $\rr$, if possible (each $v_i$ can have only one fermion).

%\begin{definition}\label{df:Vi Wi}
%	For $\ii\in\II^n$, let $\VV_\ii = \bigwedge([v_\ii]^\perp) = \Span\{v_\jj:\ii\cap\jj = \emptyset\}$ and 
%	$\WW_\ii = v_\ii \wedge \bigwedge X = \Span\{v_{\ii\jj}:\ii\cap\jj = \emptyset\} = \Span\{v_{\jj}:\ii\subset\jj\}$.
%		\SELF{$\VV_\ii = \Span\{v_\jj:\ii\subset\jj'\}$, \\
%			$\WW_\ii =  v_\ii \wedge \VV_\ii$}
%\end{definition}

\begin{example}\label{ex:a a+}
	$a_{14}^\dagger v_2 = -v_{124}$, $a_{14} v_{124} = -v_2$ and $a_{14}^\dagger v_1 = a_{14} v_{123} = 0$.
\end{example}

The following expansions of supercommutators in terms of lower order operators seem to be new.

\begin{proposition}\label{pr:scom expansion}
	For $\ii,\jj\in\II^n$, 
	we have $\scom{a_\ii,a_\jj} = \scom{a_\ii^\dagger,a_\jj^\dagger} = 0$, 
	\begin{align}
		\scom{a_\ii^\dagger,a_\jj} &= \sum\limits_{\emptyset\neq \ll \subset \ii \cap \jj} (-1)^{1 + |\ll| + |\ii\triangle\jj > \ll|} \cdot a_{\ii \backslash \ll}^\dagger \, a_{\jj \backslash \ll}, \text{ and} \label{eq:scom ai+ aj}
		\\
		\scom{a_\ii,a_\jj^\dagger} &= \sum\limits_{\emptyset\neq \ll \subset \ii \cap \jj} (-1)^{1 + |\ll| + |\ll > \ii\triangle\jj|} \cdot a_{\ii \backslash \ll} \, a_{\jj \backslash \ll}^\dagger. \label{eq:scom ai aj+}
	\end{align}
\end{proposition}
\begin{proof}
	The first two follow from \Cref{pr:exterior and interior}\ref{it:scom ee ii}.
	By \Cref{pr:exterior and interior}\ref{it:scom ie},
	\begin{align*}
	\scom{a_\jj,a_\ii^\dagger} 
	&= \sum\limits_{\emptyset \neq \ll \subset \jj}  \epsilon_{(\jj \backslash \ll) \ll} \cdot (-1)^{|\ii||\jj|+|\ll|} \cdot \ext_{v_\ii \rcontr v_\ll} \, \iota_{v_{\jj \backslash \ll}} \\
	&= \sum\limits_{\emptyset \neq \ll \subset \jj} 	 \epsilon_{(\jj \backslash \ll) \ll} \cdot (-1)^{|\ii||\jj|+|\ll|} \cdot \delta_{\ll \subset \ii} \cdot \epsilon_{(\ii \backslash \ll)\ll} \cdot a_{\ii \backslash \ll}^\dagger \, a_{\jj \backslash \ll},
	\end{align*}
	and \eqref{eq:scom ai+ aj} follows as
	$\scom{a_\ii^\dagger, a_\jj} 
	= (-1)^{1+|\ii||\jj|} \cdot \scom{a_\jj,a_\ii^\dagger}$
	and \Cref{pr:epsilon} gives
	$\epsilon_{(\jj \backslash \ll) \ll} \, \epsilon_{(\ii \backslash \ll)\ll} 
	= (-1)^{\pairs{\jj \backslash \ll}{\ll} + \pairs{\ii \backslash \ll}{\ll}}
	= (-1)^{\pairs{\ii \triangle \jj}{\ll}}$, 
	for $\ll \subset \ii \cap \jj$.
	Likewise, \eqref{eq:scom ai aj+} is obtained using \Cref{pr:exterior and interior}\ref{it:scom ei}.
\end{proof}

As $\scom{a_\ii^\dagger,a_\jj} = (-1)^{1+|\ii||\jj|} \cdot \scom{a_\ii,a_\jj^\dagger}$, the main difference between \eqref{eq:scom ai+ aj} and \eqref{eq:scom ai aj+} is that the first expansion uses $a^\dagger a$'s, the other $a a^\dagger$'s.
Below, it is the opposite, $\scom{a_\ii,a_\ii^\dagger}$ is expanded in $a^\dagger a$'s, while $\scom{a_\ii^\dagger,a_\ii}$ uses $a a^\dagger$'s.
%Again, this is a matter of convenience, to have simpler formulas.

\begin{corollary}\label{pr:scom expansion i=j} 
	$\scom{a_\ii,a_\ii^\dagger} = \sum\limits_{\jj \varsubsetneq \ii} (-1)^{|\jj|} a_{\jj}^\dagger a_{\jj}$
	and
	$\scom{a_\ii^\dagger,a_\ii} = \sum\limits_{\jj \varsubsetneq \ii} (-1)^{|\jj|} a_{\jj} a_{\jj}^\dagger$,
	for $\ii\in\II^n$.
%	\begin{equation}\label{eq:scom i i}
%		\scom{a_\ii,a_\ii^\dagger} = \sum\limits_{\jj \varsubsetneq \ii} (-1)^{|\jj|} a_{\jj}^\dagger a_{\jj},
%		\quad \text{and} \quad
%		\scom{a_\ii^\dagger,a_\ii} = \sum\limits_{\jj \varsubsetneq \ii} (-1)^{|\jj|} a_{\jj} a_{\jj}^\dagger.
%	\end{equation}
\end{corollary}
\begin{proof}
	By \eqref{eq:scom ai+ aj}, 
	$\scom{a_\ii^\dagger,a_\ii} 
	= \sum\limits_{\emptyset\neq \ll \subset \ii} (-1)^{1 + |\ll|} a_{\ii \backslash \ll}^\dagger a_{\ii \backslash \ll}
	= \sum\limits_{\jj \varsubsetneq \ii} (-1)^{1 + |\ii| + |\jj|} a_{\jj}^\dagger a_{\jj}$,
	and $\scom{a_\ii,a_\ii^\dagger} = (-1)^{1 + |\ii|} \scom{a_\ii^\dagger,a_\ii}$.
	Likewise, \eqref{eq:scom ai aj+} gives the other formula.
\end{proof}

\begin{example}\label{ex:scom}
	With $\ii\neq\jj$, we have for example:
	\begin{itemize}
		\item $\scom{a_{14}^\dagger,a_{23}} = 0$, 
			\SELF{$= \scom{a_{23},a_{14}^\dagger}$}
		as $\ii\cap\jj=\emptyset$, and so $a_{14}^\dagger a_{23} = a_{23} a_{14}^\dagger$.
		\item $\scom{a_{2347}^\dagger,a_{136}} = -a_{247}^\dagger\, a_{16}$, 
			\SELF{$= -\scom{a_{136},a_{2347}^\dagger}$. Note that $a_{247}^\dagger\, a_{16} = a_{16}\, a_{247}^\dagger$, as $[a_{247}^\dagger,a_{16}] = \scom{a_{247}^\dagger,a_{16}} = 0$.}
		as $\ii\cap\jj = 3$ and $\ii\triangle\jj = 12467$.
		\item $\scom{a_{1236}^\dagger,a_{13467}} = -a_{236}^\dagger\, a_{3467} + a_{126}^\dagger\, a_{1467} - a_{123}^\dagger\, a_{1347} + a_{26}^\dagger\, a_{467} - a_{23}^\dagger\, a_{347} + a_{12}^\dagger\, a_{147} + a_{2}^\dagger\, a_{47}$,
		as $\ii\cap\jj = 136$ and $\ii\triangle\jj = 247$.
		By \eqref{eq:scom ai aj+},
		$\scom{a_{13467},a_{1236}^\dagger} = +a_{3467}\, a_{236}^\dagger - a_{1467}\, a_{126}^\dagger + a_{1347}\, a_{123}^\dagger + a_{467}\, a_{26}^\dagger - a_{347}\, a_{23}^\dagger + a_{147}\, a_{12}^\dagger - a_{47}\, a_{2}^\dagger$. 
		Since $\scom{a_{1236}^\dagger,a_{13467}} = - \scom{a_{13467},a_{1236}^\dagger}$, 
		comparing the expansions one might think that $a_{\ii \backslash \ll}^\dagger a_{\jj \backslash \ll} = \pm a_{\jj \backslash \ll} a_{\ii \backslash \ll}^\dagger$, which is false: e.g.,
		$a_{12}^\dagger a_{147} v_{47} = 0$ but 
		$a_{147} a_{12}^\dagger v_{47} = v_2$. 
%		$a_{236}^\dagger\, a_{3467} \neq \pm a_{3467}\, a_{236}^\dagger$,
	\end{itemize}
\end{example}

\begin{example}\label{ex:scom i=j}
	Patterns are simpler when $\ii=\jj$: 
	\begin{itemize}
		\item $\scom{a_{1}^\dagger,a_{1}} = \scom{a_{1},a_{1}^\dagger} = \Id$.
		\item $\scom{a_{14}^\dagger, a_{14}} = \Id - a_{1} a_{1}^\dagger - a_{4} a_{4}^\dagger$
		and $\scom{a_{14},a_{14}^\dagger} = \Id - a_{1}^\dagger a_{1} - a_{4}^\dagger a_{4}$.
		Even though $a_\ii a_\ii^\dagger \neq \pm a_\ii^\dagger a_\ii$,
		we have $\scom{a_{14}^\dagger, a_{14}} = - \scom{a_{14},a_{14}^\dagger}$,
		since $a_{1}^\dagger a_{1} + a_{1} a_{1}^\dagger = a_{4}^\dagger a_{4} + a_{4} a_{4}^\dagger = \Id$.
		\item $\scom{a_{124}^\dagger,a_{124}} = 
		\Id - a_{1}^\dagger\, a_{1} - a_{2}^\dagger\, a_{2} - a_{4}^\dagger\, a_{4} + a_{12}^\dagger\, a_{12} + a_{14}^\dagger\, a_{14} + a_{24}^\dagger\, a_{24}$.
			\SELF{$\scom{a_{124},a_{124}^\dagger} = \Id - a_{1}\, a_{1}^\dagger - a_{2}\, a_{2}^\dagger - a_{4}\, a_{4}^\dagger + a_{12}\, a_{12}^\dagger + a_{14}\, a_{14}^\dagger +	a_{24}\, a_{24}^\dagger$. We have $\scom{a_{124}^\dagger,a_{124}} = \scom{a_{124},a_{124}^\dagger}$, though $a_1^\dagger\, a_1 \neq \pm a_1\, a_1^\dagger$, for example.}
	\end{itemize}
\end{example}

%\begin{example}\label{ex:scom}
%	$\scom{a_{14}^\dagger,a_{23}} = 0$,
%	%
%	$\scom{a_{2347}^\dagger,a_{136}} = -a_{247}^\dagger a_{16}$,
%	%
%	$\scom{a_{1236}^\dagger,a_{13467}} = -a_{236}^\dagger a_{3467} + a_{126}^\dagger a_{1467} - a_{123}^\dagger a_{1347} + a_{26}^\dagger a_{467} - a_{23}^\dagger a_{347} + a_{12}^\dagger a_{147} + a_{2}^\dagger a_{47}$. 
%\end{example}

These expansions can be helpful, but are not intuitive.
Next we give better ways to look at these supercommutators.

\subsubsection{Supercommutator \texorpdfstring{$\scom{a_\ii^\dagger,a_\ii}$}{} }

The case $\ii=\jj$ can be understood analyzing the following self-adjoint operators, as $\scom{a_\ii^\dagger,a_\ii} = n_\ii - (-1)^{|\ii|} m_\ii$.
	\OMIT{$-(-1)^{|\ii|}$ é do $\scom{\cdot,\cdot}$}

\begin{definition}\label{df:vacancy occupancy}
	$m_\ii = a_\ii a_\ii^\dagger$ and $n_\ii = a_\ii^\dagger a_\ii$ are \emph{vacancy} and \emph{occupancy operators}, respectively, for $\ii\in\II^n$.
\end{definition}

In Physics, for a single $i$, $n_i$ is a \emph{number operator} (its eigenvalues give the number $0$ or $1$ of fermions in $v_i$), and $m_i = \Id - n_i$.
	\OMIT{\ref{pr:exterior and interior}\ref{it:scom int_v}}
As \ref{it:mv nv} below shows, 
$m_\ii$ (resp. $n_\ii$) tests if all $v_i$'s with $i\in\ii$ are vacant (resp. occupied).

\begin{proposition}\label{pr:m n}
	Let $\ii,\jj\in\II^n$.
		\SELF{$m_\emptyset = n_\emptyset = \Id$ \\ $m_\ii^2=m_\ii$, $n_\ii^2=n_\ii$}
	\begin{enumerate}[i)]
%		\item $m_\ii = \PP_{\bigwedge([v_\ii]^\perp)}$ and $n_\ii = \PP_{v_\ii \wedge \bigwedge X}$.\label{it:m n proj}
%			\SELF{$m_\emptyset = n_\emptyset = \Id$ \\
%			$m_\ii^2=m_\ii$, $n_\ii^2=n_\ii$}
		\item $m_\ii v_\jj = \delta_{\ii \cap \jj = \emptyset} \, v_\jj$ and $n_\ii v_\jj = \delta_{\ii \subset \jj} \, v_\jj$. \label{it:mv nv}

		\item $m_\ii m_\jj = m_{\ii\cup\jj}$ and $n_\ii n_\jj = n_{\ii\cup\jj}$.\label{it:mimj ninj}
			\OMIT{$\VV_\ii \cap \VV_\jj = \VV_{\ii\cup\jj}$, idem $\WW$}
			\SELF{$m_\ii n_\jj v_\kk = n_\jj m_\ii v_\kk = v_\kk \delta_{\ii\subset \kk', \jj\subset \kk}$}
		\item $[m_\ii,m_\jj] = [n_\ii,n_\jj] = [m_\ii,n_\jj] = 0$.
			\OMIT{orthog projs commute}
%		\item $m_i + n_i = \Id$, for a single index $1\leq i\leq n$.\label{it:m + n single index}
%			\OMIT{$\WW_\ii = (\VV_\ii)^\perp$ for single $i$}
		\item $m_\ii = \prod_{i\in\ii} m_i$ and $n_\ii = \prod_{i\in\ii} n_i$. \label{it:m n prod}
			\OMIT{$\VV_\ii = \cap_{i\in\ii}\VV_i$ and $\WW_\ii = (\VV_\ii)^\perp = (\oplus_{i\in\ii}\VV_i)^\perp = \bigcap_{i\in\ii}\VV_i^\perp = \bigcap_{i\in\ii}\WW_i$}
		\item $m_\ii = \sum_{\jj \subset \ii} (-1)^{|\jj|} n_\jj$ and  $n_\ii = \sum_{\jj \subset \ii} (-1)^{|\jj|} m_\jj$. \label{it:m sum n}
		\item The following diagram is commutative. Stars can be left or right, but must be equal (\resp different) for equal (\resp opposite) arrows. \label{it:m n star}
		\begin{equation}\label{eq:m n star}
			\begin{tikzcd}
				\bigwedge X \arrow[d,swap,"m_\ii"] \arrow[r,"\star"] & \bigwedge X \arrow[l] \arrow[d,"n_\ii"] \\
				\bigwedge X \arrow[r] & \bigwedge X  \arrow[l,"\star"]
			\end{tikzcd}
		\end{equation}		
	\end{enumerate}
\end{proposition}
\begin{proof}
%	\emph{(\ref{it:m n proj})} Follows from \Cref{pr:triple}.
%	\emph{(\ref{it:mv nv})} $m_\ii v_\jj = a_\ii a_\ii^\dagger v_\jj = \delta_{\ii \cap \jj = \emptyset} \, a_\ii v_{\ii\jj} = \delta_{\ii \cap \jj = \emptyset} \, v_\jj$
%	and 
%	$n_\ii v_\jj = a_\ii^\dagger a_\ii v_\jj = \delta_{\ii \subset \jj} \, \epsilon_{\ii(\jj \backslash \ii)} \, a_\ii^\dagger v_{\jj \backslash \ii} = \delta_{\ii \subset \jj} \, v_\jj$.
%		\OMIT{Ou (it:P B perp), (it:BBM)}
	%
	\emph{(\ref{it:mv nv})} Follows by direct computation.
%	 with $a_\ii$ and $a_\ii^\dagger$.
		\OMIT{If $\ii\cap\jj = \emptyset$ then $a_\ii^\dagger v_\jj = v_{\ii\jj}$ and $a_\ii v_{\ii\jj} = v_\jj$ (eq:contr induced basis). Also, $a_\ii^\dagger v_{\jj} = v_\ii \wedge v_\jj = 0 \Leftrightarrow \ii\cap\jj \neq \emptyset$, 		and $a_\ii v_\jj = v_\ii \lcontr v_\jj = 0 \Leftrightarrow \ii\not\subset\jj$.}
	\emph{(\ref{it:mimj ninj}--\ref{it:m n prod})} Follow from \ref{it:mv nv}.
	\emph{(\ref{it:m sum n})} Follows from \ref{it:m n prod} and $m_\emptyset = n_\emptyset = m_i + n_i = \Id$.
	\emph{(\ref{it:m n star})} $\rH{(m_\ii v_\jj)} = \delta_{\ii\cap\jj=\emptyset} \rH{v_\jj \!} = \delta_{\ii\subset\jj'} \epsilon_{\jj\jj'}\, v_{\jj'} = \epsilon_{\jj\jj'}\, n_\ii v_{\jj'} = n_\ii(\rH{v_\jj \!})$.
	Other relations are similar.	
\end{proof}

\begin{example}\label{ex:m n}
	$m_{14} v_{2} = v_{2}$, $n_{14} v_{124} = v_{124}$ and $m_{14} v_{12} = n_{14} v_{234} = 0$.
	Also,
	$m_{14} = \Id - n_1 - n_4 + n_{14}$, and
	$n_{14} = \Id - m_1 - m_4 + m_{14}$.
\end{example}

Note how \ref{it:m sum n} relates to \Cref{pr:scom expansion i=j},
	\MAYBE{$m_\ii$, $n_\ii$ appear in $\scom{a_\ii,a_\ii^\dagger} = m_\ii - (-1)^{|\ii|} n_\ii$ and $\scom{a_\ii^\dagger,a_\ii} = n_\ii - (-1)^{|\ii|}m_\ii$ with the same signs as in these decompositions. Has meaning?}
and gives decompositions of the identity:
$\Id = m_\ii - \sum_{\emptyset\neq \jj \subset \ii} (-1)^{|\jj|} n_\jj = n_\ii - \sum_{\emptyset\neq \jj \subset \ii} (-1)^{|\jj|} m_\jj$.
This can be understood via the inclusion-exclusion principle of combinatorics and \Cref{pr:triple},\label{uso pr:triple}
which shows $m_\ii = \PP_{\Img a_\ii}$ and $n_\ii = \PP_{\Img a_\ii^\dagger}$ are orthogonal projections onto%
	\SELF{$\Img a_\ii = \Span\{v_\jj:\ii\subset\jj'\}$ \\ $\Img a_\ii^\dagger = \Span\{v_{\ii\jj}:\ii\cap\jj = \emptyset\}$ \\
	Mutually inverse isometries: \\ $a_\ii:\Img a_\ii^\dagger \rightarrow \Img a_\ii$, \\ $a_\ii^\dagger:\Img a_\ii \rightarrow \Img a_\ii^\dagger$}
	\MAYBE{$\Img \ext_B = B \wedge \bigwedge([B]^\perp)$ 
	for any blade $B$, and $\Img \iota_B = \bigwedge ([B]^\perp)$ if $B\neq 0$. Generaliza com $\isp{M}$?}
%\begingroup\setlength\belowdisplayskip{0pt}
\begin{align*}
	\Img a_\ii &= (\ker a_\ii^\dagger)^\perp = \bigwedge([v_\ii]^\perp) = \Span\{v_\jj:\jj\in\II^n,\ii\cap\jj = \emptyset\}, \text{ and} \\
	\Img a_\ii^\dagger &= (\ker a_\ii)^\perp 
	= v_\ii \wedge \bigwedge([v_\ii]^\perp) = \Span\{v_{\jj}:\jj\in\II^n,\ii\subset\jj\}.
\end{align*}
%\endgroup

\begin{example}
	$\scom{a_{14}^\dagger,a_{14}} = n_{14} - m_{14}  = \PP_{\jj\supset 14} - \PP_{\jj\cap 14 = \emptyset}$,
	where $\PP_{\mathbf{P}(\jj)}$ is the orthogonal projection onto $\Span\{v_\jj : \mathbf{P}(\jj) \text{ is true}\}$.
		\SELFL{$\scom{a_{124}^\dagger,a_{124}} = n_{124} + m_{124}  = \PP_{\jj\supset 124} + \PP_{\jj\cap 124 = \emptyset}$}
	The expansions of $m_{14}$ and $n_{14}$ in \Cref{ex:m n}
	mean
	$\Id = \PP_{\jj\cap 14=\emptyset} + \PP_{\jj\supset 1} + \PP_{\jj\supset 4} - \PP_{\jj\supset 14}
	= \PP_{\jj\supset 14} + \PP_{\jj\cap 1=\emptyset} + \PP_{\jj\cap 4=\emptyset} - \PP_{\jj\cap 14=\emptyset}$.
	Similar interpretations can be obtained in Fig.\,\ref{fig:Venn 2 3} for $\ii=124$.
	Compare with \Cref{ex:scom i=j}.
\end{example}

\begin{figure}[t]
	\centering
	\begin{subfigure}[b]{0.405\textwidth}
		\includegraphics[width=\textwidth]{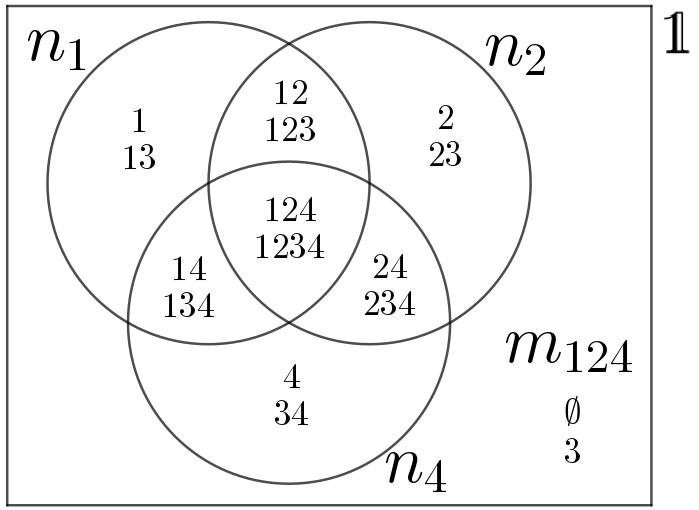}
		\caption{$\Id = m_{124} + n_1 + n_2 + n_4 - n_{12} - n_{14} - n_{24} + n_{124}$}
		\label{fig:Venn 3}
	\end{subfigure}
	\hspace{0.05\textwidth}
	\begin{subfigure}[b]{0.4\textwidth}
		\includegraphics[width=\textwidth]{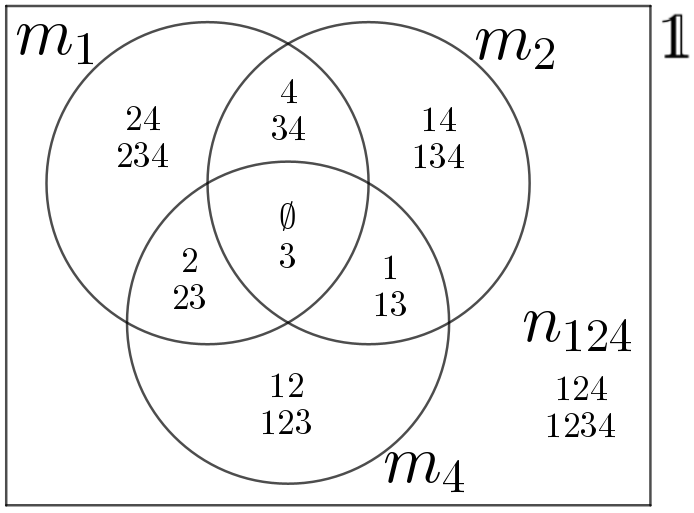}
		\caption{$\Id = n_{124} + m_1 + m_2 + m_4 - m_{12} - m_{14} - m_{24} + m_{124}$}
		\label{fig:Venn 3b}
	\end{subfigure}
	\caption{Decompositions of $\Id$ via inclusion-exclusion principle, for $\ii=124$ in $X=\Span\{v_1,\ldots,v_4\}$. 
	(a) $n_1$, $n_2$, $n_4$ project onto subspaces of $\bigwedge X$ spanned by $v_{\jj}$'s with $\jj$ in the corresponding circles, $n_{12}$, $n_{14}$, $n_{24}$, $n_{124}$ project in their intersections, and $m_{124}$ in their complement. 
	(b) Alternative decomposition, switching the roles of $m$'s and $n$'s. 
	The symmetry between the diagrams reflects the duality given by \eqref{eq:m n star}.
	}
	\label{fig:Venn 2 3}
\end{figure}

\begin{corollary}\label{pr:scom m n}
	For $\ii,\jj\in\II^n$,
	$\scom{a_\ii^\dagger,a_\ii} v_\jj = 
	\begin{cases}
		v_\jj &\text{if } \emptyset \neq \ii\subset\jj,\\
		(-1)^{|\ii|+1} v_\jj &\text{if } \emptyset \neq \ii\subset\jj', \\
		0 &\text{otherwise}.
	\end{cases}$ \label{it:scom v}
	\OMITR{\ref{pr:m n}\ref{it:mv nv} \\ 
		$\ii=\emptyset$ dá $v_\jj-v_\jj$}
\end{corollary}
%\begin{proof}
%	Follows from \Cref{pr:m n}\ref{it:mv nv}.
%\end{proof}

For $\scom{a_\ii,a_\ii^\dagger} = (-1)^{|\ii|+1} \cdot \scom{a_\ii^\dagger,a_\ii}$, the first two conditions are swapped.

\subsubsection{Supercommutator \texorpdfstring{$\scom{a_\ii^\dagger,a_\jj}$}{} } \label{sc:General supercommutator}

The case $\ii\neq \jj$ is best understood by analyzing how $\scom{a_\ii^\dagger,a_\jj}$ acts on the $v_\kk$'s.
This has not been described before, as far as we know.
See \Cref{tab:symbols} for notation.
%Recall that $\ord{\rr}$ means $\rr \in \MM^q$ with its indices in increasing order.

\begin{theorem}
	Separate the indices of $\ii,\jj,\kk\in\II^n$ into pairwise disjoint  $\aa,\bb,\cc,\dd,\ee,\xx,\yy\in\II^n$ 
		\SELF{works with $\MM^n$}
	as in the Venn diagram of Fig.\,\ref{fig:venn}. Then
		\SELF{$\sacom{a_{\ii}^\dagger,a_{\jj}}$ has $\delta_{\cc=\emptyset} + \delta_{\dd=\emptyset}$}
		\SELF{$\ii = \jj$ gives \ref{pr:scom m n}\ref{it:scom v}: $\xx=\yy=\aa=\bb=\emptyset$. If $\dd=\emptyset$, $\ii=[\cc]$, $\kk=[\cc\ee]$, $\ii\subset\kk$. If $\cc=\emptyset$, $\ii=[\dd]$, $\kk=[\ee]$, $\ii\subset\kk'$. If $\cc,\dd\neq\emptyset$, $\ii=[\cc\dd]$, $\kk=[\cc\ee]$, $\ii\not\subset \kk,\kk'$. \\
		It is only way to have $\scom{a_\ii^\dagger,a_\jj} v_\kk =  \pm v_\kk$}
	\begin{equation}\label{eq:scom ai+ aj vk}
		\scom{a_\ii^\dagger,a_\jj} v_\kk = \delta_{\xx\yy=\emptyset} \cdot (\delta_{\dd=\emptyset} - \delta_{\cc=\emptyset}) \cdot (-1)^{|\dd|+\pairs{\aa\bb}{\dd\ee}} v_{\ord{\aa\cc\ee}}.
	\end{equation}
\end{theorem}
\begin{proof}		
	$\scom{a_\ii^\dagger,a_\jj} v_\kk =  
	a_\ii^\dagger a_\jj v_\kk - (-1)^{|\ii| |\jj|} a_\jj a_\ii^\dagger v_\kk$, 
	and both terms vanish if $\xx = \jj \backslash (\ii\cup\kk)\neq\emptyset$ 
	or $\yy = \ii \cap (\kk \backslash \jj) \neq \emptyset$.
	If $\xx = \yy=\emptyset$ then
	$\ii = \ord{\aa\cc\dd}$, $\jj = \ord{\bb\cc\dd}$ and $\kk = \ord{\bb\cc\ee}$. 
	We have
	$a_{\aa\cc\dd}^\dagger\, a_{\bb\cc\dd}\, v_{\bb\cc\ee} 
	= \delta_{\dd=\emptyset}\, a_{\aa\cc}^\dagger\, v_\ee
	= \delta_{\dd=\emptyset}\, v_{\aa\cc\ee}$,
	while 
	$a_{\bb\cc\dd}\, a_{\aa\cc\dd}^\dagger\, v_{\bb\cc\ee}
	= \delta_{\cc=\emptyset}\, a_{\bb\dd}\, v_{\aa\dd\bb\ee} 
	= \delta_{\cc=\emptyset}\cdot (-1)^{|\aa||\bb\dd|+|\bb||\dd|} v_{\aa\ee}
	= \delta_{\cc=\emptyset}\cdot  (-1)^{|\ii| |\jj|+|\dd|} v_{\aa\cc\ee}$.
	Thus
		\OMIT{$a_{[\rr]} = \iota_{v_{[\rr]}} = 	\iota_{\epsilon_\rr v_{\rr}}  = \epsilon_\rr a_{\rr}$, likewise $a_{[\rr]}^\dagger = \epsilon_\rr a_{\rr}^\dagger$}
	$\scom{a_\ii^\dagger,a_\jj} v_\kk 
	= \epsilon_{\aa\cc\dd} \epsilon_{\bb\cc\dd}\, \epsilon_{\bb\cc\ee}  \scom{a_{\aa\cc\dd}^\dagger,a_{\bb\cc\dd}} v_{\bb\cc\ee}
	= \epsilon_{\aa\cc\dd}\, \epsilon_{\bb\cc\dd}\, \epsilon_{\bb\cc\ee}\, \epsilon_{\aa\cc\ee} \cdot (\delta_{\dd=\emptyset} - \delta_{\cc=\emptyset})  \cdot (-1)^{|\dd|}  v_{\ord{\aa\cc\ee}}$.
	The result follows since \Cref{pr:epsilon} gives
	$\epsilon_{\aa\cc\dd}\, \epsilon_{\bb\cc\dd}\, \epsilon_{\bb\cc\ee}\, \epsilon_{\aa\cc\ee}
	= \epsilon_{\aa\cc}^2\, \epsilon_{\cc\dd}^2\, \epsilon_{\bb\cc}^2\, \epsilon_{\cc\ee}^2\, \epsilon_{\aa\dd}\, \epsilon_{\bb\dd}\, \epsilon_{\bb\ee}\, \epsilon_{\aa\ee} 
	= (-1)^{\pairs{\aa}{\dd} + \pairs{\bb}{\dd} + \pairs{\bb}{\ee} + \pairs{\aa}{\ee}}
	= (-1)^{\pairs{\aa\bb}{\dd\ee}}$.
\end{proof}

\begin{figure}
	\centering
	\includegraphics[width=0.25\linewidth]{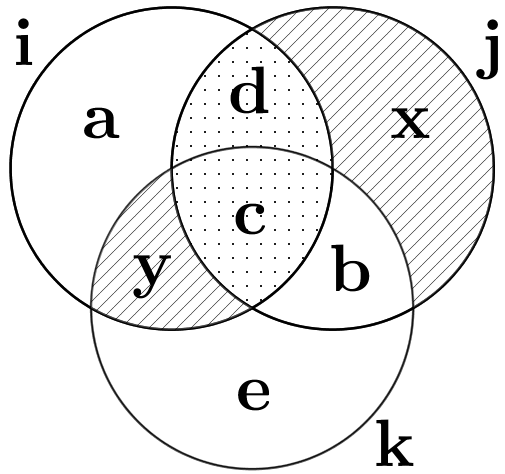}
	\caption{Partitioning the indices of $\ii,\jj,\kk\in\II^n$ as above,
		if $\xx$, $\yy$, and only one of $\cc$ or $\dd$, are empty then $\scom{a_ \ii^\dagger,a_\jj} v_\kk = \pm v_{\aa\cc\ee} = \pm v_{[\kk \cup ( \ii \backslash  \jj)] \backslash ( \jj \backslash  \ii)}$,
		otherwise $\scom{a_\ii^\dagger,a_\jj} v_\kk = 0$.}
	\label{fig:venn}
\end{figure}

%Note that $\ord{\aa\cc\ee} = (\kk \cup ( \ii \backslash  \jj)) - ( \jj \backslash  \ii)$  when $\xx\yy = \emptyset$.

\begin{corollary}\label{pr:scom ai aj+ ek = 0}
	$\scom{a_\ii^\dagger,a_\jj} v_\kk = 0$ in the following cases, and only in them:
	\begin{enumerate}[i)]
		\item $\jj \backslash (\ii \cup \kk) \neq \emptyset$. \label{it:x}
		\item $(\ii \cap \kk) \backslash  \jj \neq \emptyset$. \label{it:y}
		\item $(\ii \cap \jj) \backslash \kk \neq \emptyset$ and $\ii \cap \jj \cap \kk \neq \emptyset$. \label{it:c u}
		\item $\ii \cap \jj = \emptyset$. \label{it:c u empty}
	\end{enumerate}
\end{corollary}
\begin{proof}
	\emph{(\ref{it:x})} $\xx\neq\emptyset$.
	\emph{(\ref{it:y})} $\yy\neq\emptyset$.
	\emph{(\ref{it:c u})} $\dd\neq\emptyset$ and $\cc\neq\emptyset$. \emph{(\ref{it:c u empty})} $\dd=\cc=\emptyset$.
\end{proof}

\begin{corollary}
	$\scom{a_\ii^\dagger,a_\jj} = 0 \Leftrightarrow \ii\cap\jj=\emptyset$.
\end{corollary}
\begin{proof}
	If $\ii\cap\jj \neq \emptyset$, none of the above conditions holds for $\kk =  \jj \backslash  \ii$.
		\OMIT{\ref{pr:scom ai aj+ ek = 0} \\ $ \jj \backslash  ( \ii \cup ( \jj \backslash   \ii)) =  \jj \backslash  ( \jj\cup \ii) = \emptyset$, \\ $( \ii \cap ( \jj \backslash  \ii)) -  \jj = \emptyset$, \\ $\ii \cap \jj \cap ( \jj \backslash  \ii) = \emptyset$}
\end{proof}

\begin{example}\label{ex:scom i j}
	Compare the following with \Cref{ex:scom}.
	\begin{itemize}
		\item $\scom{a_{2347}^\dagger,a_{136}} v_{1356} = (1-0) \cdot (-1)^{0+\pairs{12467}{5}} v_{23457} = v_{23457}$, 
		since
		$\xx=\yy=\dd=\emptyset$, $\aa=247$, $\bb=16$, $\cc=3$ and $\ee=5$.
		\item $\scom{a_{1236}^\dagger,a_{13467}} v_{457} = (0-1) \cdot (-1)^{3+\pairs{247}{1356}} v_{25} = -v_{25}$, 
		since
		$\xx=\yy=\cc=\emptyset$, $\aa=2$, $\bb=47$, $\dd=136$ and $\ee=5$.
		\item $\scom{a_{12}^\dagger,a_{234}} v_{45} = \scom{a_{123}^\dagger,a_{24}} v_{2345} = \scom{a_{123}^\dagger,a_{234}} v_{345} = \scom{a_{14}^\dagger,a_{23}} = 0$, as each corresponds to a case of \Cref{pr:scom ai aj+ ek = 0}.
	\end{itemize}
\end{example}

To see how \eqref{eq:scom ai+ aj} relates to \eqref{eq:scom ai+ aj vk},
note that, for $\emptyset\neq \ll \subset \ii \cap \jj = \ord{\cc\dd}$,
we have $a_{\ii \backslash \ll}^\dagger a_{\jj \backslash \ll} v_\kk = a_{\ord{\yy\aa\cc\dd \backslash \ll}}^\dagger \, a_{\ord{\xx\bb\cc\dd \backslash \ll}} \, v_{\ord{\yy\bb\cc\ee}} = 0$ 
unless $\xx = \yy=\emptyset$ and $\dd\subset\ll$,
in which case
$\ll=\ord{\dd\mm}$ for $\mm\subset \cc$,
and 
$a_{\ii \backslash \ll}^\dagger a_{\jj \backslash \ll} v_\kk 
= a_{\ord{\aa\cc \backslash \mm}}^\dagger\, a_{\ord{\bb\cc \backslash \mm}}\, v_{\ord{\bb\cc\ee}} 
= \epsilon_{\aa\cc \backslash \mm} \, \epsilon_{\bb\cc \backslash \mm} \, \epsilon_{\bb\cc\ee} \cdot a_{\aa\cc \backslash \mm}^\dagger\, a_{\bb\cc \backslash \mm}\, v_{(\bb\cc \backslash \mm)\mm\ee} 
= \epsilon_{\aa\cc \backslash \mm} \, \epsilon_{\bb\cc \backslash \mm} \, \epsilon_{\bb\cc\ee} \, \epsilon_{\aa\cc\ee} \cdot v_{\ord{\aa\cc\ee}} 
= (-1)^{\pairs{\aa\bb}{\cc \backslash \mm} + \pairs{\aa\bb}{\cc\ee} } \, v_{\ord{\aa\cc\ee}}
= (-1)^{\pairs{\aa\bb}{\mm\ee}} \, v_{\ord{\aa\cc\ee}}$.
%
%\begin{align*}
%a_{\ii \backslash \ll}^\dagger a_{\jj \backslash \ll} v_\kk 
%&= a_{\ord{\aa\cc \backslash \mm}}^\dagger\, a_{\ord{\bb\cc \backslash \mm}}\, v_{\ord{\bb\cc\ee}} \\
%&= \epsilon_{\aa\cc \backslash \mm} \, \epsilon_{\bb\cc \backslash \mm} \, \epsilon_{\bb\cc\ee} \, a_{\aa\cc \backslash \mm}^\dagger\, a_{\bb\cc \backslash \mm}\, v_{\bb(\cc \backslash \mm)\mm\ee} \\
%&= (-1)^{\pairs{\aa\bb}{\cc \backslash \mm}} \, \epsilon_{\bb\cc\ee} \, \epsilon_{\aa\cc\ee} \, v_{\ord{\aa\cc\ee}} \\
%&= (-1)^{\pairs{\aa\bb}{\cc \backslash \mm} + \pairs{\aa\bb}{\cc\ee} } \, v_{\ord{\aa\cc\ee}}.
%\end{align*}
Then \eqref{eq:scom ai+ aj} gives
\begin{align*}
	\scom{a_\ii^\dagger,a_\jj} v_\kk 
	&= \sum_{\mm \subset \cc, \dd\mm\neq\emptyset} (-1)^{1+|\dd|+|\mm|+\pairs{\aa\bb}{\dd\mm} + \pairs{\aa\bb}{\mm\ee}} \, v_{\ord{\aa\cc\ee}} \\
	&= (-1)^{|\dd|+|\aa\bb>\dd\ee|}\, v_{\ord{\aa\cc\ee}} \cdot \sum_{\mm \subset \cc, \dd\mm\neq\emptyset} (-1)^{1+|\mm|}.
\end{align*}
%as
%$\pairs{\aa\bb}{\dd\mm} + \pairs{\aa\bb}{\cc \backslash \mm} + \pairs{\aa\bb}{\cc\ee}
%= \pairs{\aa\bb}{\cc\dd} + \pairs{\aa\bb}{\cc\ee}
%=\pairs{\aa\bb}{\dd\ee} + 2 \pairs{\aa\bb}{\cc}$.
If $\cc=\dd=\emptyset$, the sum has no terms.
If $\cc,\dd\neq\emptyset$, it cancels out, as the number of subsets $\mm\subset\cc$ with $|\mm|$ even or odd is the same.
If $\cc\neq\emptyset$ and $\dd=\emptyset$ then $\mm=\emptyset$ is excluded from the sum, one odd $|\mm|$ term does not cancel out, and
$\scom{a_\ii^\dagger,a_\jj} v_\kk = (-1)^{|\dd|+|\aa\bb>\dd\ee|}\, v_{\ord{\aa\cc\ee}}$.
If $\cc=\emptyset$ and $\dd\neq\emptyset$, the only term is $\mm=\emptyset$, 
so 
$\scom{a_\ii^\dagger,a_\jj} v_\kk = -(-1)^{|\dd|+\pairs{\aa\bb}{\dd\ee}} v_{\ord{\aa\cc\ee}}$.

\section*{Acknowledgments}

The author would like to thank Dr. Leo Dorst of the University of Amsterdam for his comments and suggestions.

%\vspace{6pt}
%
%\noindent
%\textbf{Note.} This article has been posted to the arXiv e-print repository, with the identifier arXiv:..........

%\bibliographystyle{amsplain} 
%\bibliography{../../../Bibliografia_Linear_Geometry/Linear_Geometry}

\providecommand{\bysame}{\leavevmode\hbox to3em{\hrulefill}\thinspace}
\providecommand{\MR}{\relax\ifhmode\unskip\space\fi MR }
\providecommand{\MRhref}[2]{%
	\href{http://www.ams.org/mathscinet-getitem?mr=#1}{#2}
}
\providecommand{\href}[2]{#2}

\end{document}